\documentclass{article}
\usepackage{graphics}
\usepackage{ifthen}
\usepackage{amsmath,amsfonts,amssymb,a4,enumerate,epsfig,theorem,psfrag,subfigure}
\usepackage[latin1]{inputenc}

\newcommand{\Ran}{{\operatorname{Ran}}}

\newcommand{\RR}{{\mathbb{R}}}

\newcommand{\Xh}{{X^h}}
\newcommand{\Yh}{{Y^h}}
\newcommand{\Fh}{{F^h}}

\newcommand{\Vp} {\ensuremath{V_{\scriptscriptstyle X}}}%
\newcommand{\Vd} {\ensuremath{V_{\scriptscriptstyle Y}}}%
\newcommand{\Wp} {\ensuremath{W_{\scriptscriptstyle X}}}%
\newcommand{\Wd} {\ensuremath{W_{\scriptscriptstyle Y}}}%

\newcommand{\Vph} {\ensuremath{V_{\scriptscriptstyle X}^{h}}}%
\newcommand{\Vdh} {\ensuremath{V_{\scriptscriptstyle Y}^{h}}}%
\newcommand{\Wph} {\ensuremath{W_{\scriptscriptstyle X}^{h}}}%
\newcommand{\Wdh} {\ensuremath{W_{\scriptscriptstyle Y}^{h}}}%

\newcommand{\wwd} {\ensuremath{w_{\scriptscriptstyle Y}}}%
\newcommand{\vvd} {\ensuremath{v_{\scriptscriptstyle Y}}}%
\newcommand{\Wpt} {\ensuremath{\tilde{W}_{\scriptscriptstyle {X}}}}%
\newcommand{\Wdt} {\ensuremath{\tilde{W}_{\scriptscriptstyle {Y}}}}%
\newcommand{\Pd} {\ensuremath{P_{\scriptscriptstyle Y}}}%
\newcommand{\Pp} {\ensuremath{P_{\scriptscriptstyle X}}}%
\newcommand{\Qd} {\ensuremath{Q_{\scriptscriptstyle Y}}}%
\newcommand{\Qp} {\ensuremath{Q_{\scriptscriptstyle X}}}%

\newcommand{\optpar} [2] []
  {\ensuremath{#2\ifempty{#1} {} {\!\left(#1\right)}}}
\newcommand{\optsub} [2] []
  {\ensuremath{#2\ifempty{#1} {} {_{#1}}}}
\newcommand{\optsuper} [2] []
  {\ensuremath{#2\ifempty{#1} {} {^{#1}}}}

\newlength{\boxwd}
\newlength{\boxht}

\newcommand{\genstack} [3] [c] {\settodepth{\boxht} {#2}\settowidth{\boxwd} {#2}%
\addtolength{\boxht} {2pt}\makebox[0pt] [l] {#2}%
\raisebox{-\boxht} {%
 \raisebox{-\height} {\makebox[\boxwd] [#1] {#3}}}%
}
\newcommand{\stack} [3] [c]{
 \genstack[#1] {\ensuremath{#2}} {\ensuremath{#3}}
}

\newcommand{\stacksub} [3] [c]{
 \stack[#1]  {#2} {\scriptstyle{#3}}
}

\newcommand{\opmapsto}[1] {\stackrel{#1}{\mapsto}}

\newcommand{\Range} {\operatorname{Ran}}

\newcommand{\Kernel} {\operatorname{Ker}}

\newcommand{\linspan} [1] {\operatorname{Span}\{#1\}}

\newcommand{\st} {\,:\,}

\newcommand{\df} {\ensuremath{:=}}

\newcommand{\dx} [1] [x] {\,d#1}

\newcommand{\ifempty} [1] {\ifthenelse{\equal {#1} {} }}

\newcommand\iprod [3] [] {\optsub[\scriptscriptstyle#1]{\ensuremath{\langle#2,#3\rangle}}}
\newcommand\negprod [2]   {\iprod[-1]{#1}{#2}}
\newcommand\zeroprod [2]   {\iprod[0]{#1}{#2}}
\newcommand\oneprod [2]   {\iprod[1]{#1}{#2}}
\newcommand{\spectrum} [1] {\ensuremath{\sigma}{\left(#1\right)}}

\newcommand{\dOmega}{\ensuremath{\partial\Omega}}

\newcommand{\Lp} [2] [] {\optpar[#1] {L^{#2}}}                   
\newcommand{\Lone} [1] [] {\Lp[#1]{1}}
\newcommand{\Ltwo} [1] [] {\Lp[#1]{2}}
\newcommand{\Linfty} [1] [] {\Lp[#1]{\infty}}
\newcommand{\Hk} [2] [] {\optpar[#1] {H^{#2}}}
\newcommand{\Hone} [1] [] {\Hk[#1]{1}}
\newcommand{\Hnegone} [1] [] {\Hk[#1]{-1}}
\newcommand{\Htwo} [1] [] {\Hk[#1]{2}}
\newcommand{\Hzerok} [2] [] {\optpar[#1] {H_0^{#2}}}
\newcommand{\Hzeroone} [1] [] {\Hzerok[#1]{1}}
\newcommand{\Hzerotwo} [1] [] {\Hzerok[#1]{2}}
\newcommand{\Wk} [2] [] {\optpar[#1] {W^{#2}}}

\newcommand{\Wzerok} [2] [] {\optpar[#1] {W_0^{#2}}}

\newcommand{\Ck} [2] [] {\optpar[#1] {C^{#2}}}
\newcommand{\Cone} [1] [] {\Ck[#1]{1}}

\newcommand{\Cinfty} [1] [] {\Ck[#1]{\infty}}
\newcommand{\calK} {\ensuremath{\mathcal{I}}}

\newcommand{\Rn} [1] [n]{\ensuremath{\mathbb{R}^{#1}}}
\newcommand{\R} {\ensuremath{\mathbb{R}}}

\newcommand{\lap} {\Delta} 

\newcommand{\seminorm} [1] [\,] {\ensuremath{|{#1}|}}    
\newcommand{\norm} [1] [\,] {\ensuremath{|\!|{#1}|\!|}}  

\newcommand{\semnrm} [2] [\,]
  {\ensuremath{\seminorm[{#1}]_{\scriptscriptstyle#2}}}
\newcommand{\nrm} [2] [\,]%
  {\ensuremath{|\!|{#1}|\!|_{{\scriptscriptstyle#2}}}}

\newcommand{\Lonenrm} [2] []%
  {\nrm[{#2}] {\Lone[#1]}}
\newcommand{\Ltwonrm} [2] []%
  {\nrm[{#2}] {\Ltwo[#1]}}
\newcommand{\Linftynrm} [2] []%
  {\nrm[{#2}] {\Linfty[#1]}}
\newcommand{\Honenrm} [2] []%
  {\nrm[{#2}] {\Hone[#1]}}
\newcommand{\Htwonrm} [2] []%
  {\nrm[{#2}] {\Htwo[#1]}}

\newcommand{\negnrm} [1] []%
  {\nrm[{#1}] {-1}}
\newcommand{\zeronrm} [1] []%
  {\nrm[{#1}] {0}}
\newcommand{\onenrm} [1] []%
  {\nrm[{#1}] {1}}
\newcommand{\twonrm} [1] []%
  {\nrm[{#1}] {2}}
\newcommand{\inftynrm} [1] []%
  {\nrm[{#1}] {\infty}}

\newcommand{\onesemnrm} [1] []%
  {\semnrm[{#1}] {1}}

\newcommand{\Honesemnrm} [2] []%
  {\semnrm[{#2}] {\Hone[#1]}}
\newcommand{\Htwosemnrm} [2] []%
  {\semnrm[{#2}] {\Htwo[#1]}}

\newtheorem{lem} {Lemma}
\newtheorem{prop} {Proposition}
\newtheorem{theo} {Theorem}

\def\qed{\unskip\nobreak\hfil\penalty50\hskip1.75em\null\nobreak\hfil
$\blacksquare$ {\parfillskip=0pt \finalhyphendemerits=0 \par}\medbreak}

\newcommand{\fe}{\ensuremath{\mathcal{P}_1}}

\newcommand{\uh} [1] [u] {\ensuremath{{#1}^{h}}}

\newcommand{\unbar} [1] [u] {\ensuremath{\underline{\mathrm{#1}}}}
\newcommand{\uhat} [1] [u] {\ensuremath{\hat{\mathrm{#1}}}}
\newcommand{\mat} [1] [M] {\ensuremath{\mathrm{#1}}}

\newcommand{\Qph} {\mat[\Qp^h]}

\newcommand{\Pph} {\mat[\Pp^h]}

\title{Numerical analysis of semilinear elliptic equations \\ with finite spectral interaction}
\author{Jos\'e Teixeira Cal Neto and Carlos Tomei }

\begin{document}
\maketitle

\begin{abstract}
We present an algorithm to solve $- \lap u - f(x,u) = g$ with Dirichlet boundary conditions in a bounded domain $\Omega$. The nonlinearities are non-resonant and have finite spectral interaction: no eigenvalue of $-\lap_D$ is an endpoint of $\overline{\partial_2f(\Omega,\RR)}$, which in turn only contains a finite number of eigenvalues. The algorithm is based in ideas used by Berger and Podolak to provide a geometric proof of the Ambrosetti-Prodi theorem and advances work by Smiley and Chun for the same problem.
\end{abstract}

\medbreak


\smallbreak

{\noindent\bf MSC-class:} 35B32, 35J91, 65N30.

\section{Introduction}\label{sec:intro}

In this paper, we describe  a numerical algorithm to solve
\begin{equation}\label{eq:problem}
  - \lap u(x) - f(x,u(x)) = g(x), \quad u|_{\partial \Omega} = 0.
\end{equation}
Here, the nonlinearity $f$ is
an \emph{appropriate} function, to be defined later, and the domain $\Omega \in \RR^n$ is open, bounded, connected and has Lipschitz boundary $\dOmega$.

Theoretical tools go hand in hand with numerical methods.
Local behavior at regular points concerns both the inverse function theorem and Newton's inversion algorithm. Homotopy arguments such as degree theory go along well with continuation methods.
The celebrated mountain pass lemma (\cite{rabinowitz:1986}) is the starting point of an algorithm presented in \cite{choimckenna:1993}. More recently, ideas used in computer assisted proofs were combined with the topological toolbox with striking effect by Breuer, McKenna and Plum (\cite{bmp:2003}). In this paper, we explore numerically a global Lyapunov-Schmidt decomposition.

More precisely, we consider a class of \Cone\ maps $F: X \to Y$  between Banach spaces. Split $X = \Wp \oplus \Vp$ and $Y = \Wd \oplus \Vd$ into closed \emph{horizontal} and \emph{vertical} subspaces. Define complementary projections $\Pd, \Qd:Y \to Y$ so that $\Ran\  \Pd = \Wd$ and $\Ran\  \Qd = \Vd$.  A map $F$ is \emph{flat} if, for each $x \in X$, $\Pd \circ F : x + \Wp \to \Wd$ is a diffeomorphism. Here $x + \Wp$ is the affine horizontal subspace obtained by translating $\Wp$ by $x$.

This stringent hypothesis gives rise to substantial geometric structure. Images under $F$ of affine horizontal subspaces, \emph{sheets}, intercept transversally affine vertical subspaces $y + \Vd, y \in Y$, at a unique point. Preimages under $F$ of affine vertical subspaces, \emph{fibers}, are submanifolds of $X$ diffeomorphic to $\Vp$. Indeed, $X$ is foliated by fibers, and each fiber intersects transversally each affine horizontal subspace at a unique point.

The bifurcation equations related to the decomposition for $F(x) = y$ are
\[ \Pd F(w + v) = \Pd y, \quad \Qd F(w + v) = \Qd y, \quad
w \in \Wp, \ v \in \Vp \]
and the first equation, by flatness, admits a unique solution $w(v)$ for each fixed $v$. Said differently, given $y \in Y$, $w(v)$ is the unique point of the fiber $F^{-1}(y + \Vd)$ in the affine horizontal space $v + \Wp$. Clearly, the fiber through $w(v)$ contains all solutions of $F(x) = y$.

In a nutshell, the algorithm first computes $w(v)$ based on the finite element method. The search for solutions of $F(x) = y$ then reduces to inverting a (computable) map between isomorphic finite dimensional subspaces $\Vp$ and $\Vd$.

We use piecewise linear finite elements. For the sake of sparsity, we exploit the decompositions of spaces $X$ and $Y$. Indeed, from flatness, a large part of the derivative $DF(u)$, taking \Wp\ to \Wd, is invertible. We extend this isomorphism to one from $X$ to $Y$ which is especially simple to code. The search for $w(v)$ becomes then a standard continuation method between given affine horizontal subspaces, with the advantage that computations are performed in the full spaces $X$ and $Y$.

The restriction of $F$ to a fixed fiber $\alpha$ now can be computed by a predictor-corrector algorithm. In more detail, $\Vp$ parametrizes $\alpha$ and, given $x \in \alpha$, a point $x + v, v\in \Vp$ corresponds to a unique point $\tilde{x} \in \alpha$ with the same height (i.e., $\Qp (x+v) = \Qp \tilde{x}$), obtained from $x+v$ by the same continuation method used to compute $w(v)$. We are then left with inverting a (constructible) function between finite-dimensional spaces.

The constructions rely on the assumption that the projections $\Qp$ and $\Qd$ are computable. In order to avoid unnecessary abstraction, we present the algorithm for the special case related to equation \eqref{eq:problem}. This allows us to discuss some implementation issues.
We consider the map $F(u)(x) = -\lap u(x) - f(x,u(x))$ between Sobolev spaces
$X = \Hzeroone[\Omega]$ and $Y = \Hnegone[\Omega] \simeq \Hzeroone[\Omega]$.
The nonlinearity $f$ is assumed to be \emph{appropriate}: \[
f\in \Cone[{\Omega}\times\RR],\quad \Ltwonrm{f(\cdot,0)}<\infty\quad\textrm{and}\quad \inftynrm[\partial_2 f]<\infty.
 \]
Here $\partial_2$ is the partial derivative with respect to the second variable. Let $I=[a,b]$ be an interval  containing $\overline{\partial_2f(\Omega,\RR)}$.
Set $\Vp = \Vd$ to be the direct sum of the (maximal) invariant subspaces associated to the eigenvalues  in $[a,b]$ of $-\lap_D$. Finally, let $\Wp = \Vp^\perp$ and $\Wd = \Vd^\perp$ in $X$ and $Y$ respectively. As we shall see in Section \ref{sec:estimates}, $F: X \to Y$ is flat with respect to this decomposition.

The search for $w(v)$ is robust and globally stable: errors self-correct in the spirit of Newton-type iterations, and the linear operators which require inversion are both uniformly bounded and uniformly coercive.
Searching for solutions in the fiber is not necessarily an easy task. When $\dim \Vp = \dim \Vd = 1$, one needs to invert a function from $\RR$ to $\RR$, and root solvers abound. For higher dimensions, matters are harder. For the two-dimensional case, we present an example in Section \ref{subsec:int2lams}.

The history of semilinear elliptic theory, together with some computational aspects, is very well described in \cite{bmp:2003}.
Here we emphasize some techniques which are relevant to our text. A good introduction to computer assisted proofs in a related context is \cite{plum:2009}.

Hammerstein (\cite{hammerstein:1930}) and Dolph (\cite{dolph:1949}) showed that, if $\overline{\partial_2f(\Omega,\RR)}$ does not contain any eigenvalue of $-\lap_D$, the map $F: X \to Y$ is a (global) diffeomorphism. For the choice $\Vp = \Vd= \{0\}$, $F$ is trivially flat. Numerical inversion might proceed by standard continuation methods, requiring inversion of $-\lap v(x) - \partial_2f(x,u(x)) v(x) = h(x)$, with Dirichlet boundary conditions.

In the autonomous case $f(x,u)=f(u)$, Ambrosetti and Prodi (\cite{ambrosetti:1972}) presented a thorough analysis of a situation in which $F(u)=g$ admits multiple solutions.
Their result, immediately amplified by Manes and Micheletti (\cite{manes:1973}), essentially states that if $f$ is convex and $\overline{f'(\RR)}$ only contains the smallest eigenvalue $\lambda_1$ of $-\lap_D$, then $g$ can only have 0, 1 or 2 preimages.
Later, the Ambrosetti-Prodi map $F$ was given a novel geometric description by Berger and Podolak (\cite{berger:1974}): for the  Lyapunov-Schmidt decomposition $\Vp = \Vd = \linspan{\varphi_1},$ they showed that $F$ is flat. Here $\varphi_1$ is a positive eigenfunction associated to $\lambda_1$.
Their proof uses a coercive bound $||DF(w+v)|| \ge C ||w||$, uniform in $v$. They then showed that each fiber $\alpha$, the preimage under $F$ of a vertical affine subspace, is a differentiable curve. Moreover, the restriction of $F$ to each $\alpha$ becomes $s \mapsto - s^2$, after a change of variable. Since fibers foliate $X$, $F$ is a global fold: (global) changes of variables from $X$ and $Y$
to a common space $Z$ convert $F$ into $\tilde{F}: Z \to Z$ where $Z = V \oplus \tilde{W}$  given by $\tilde{F}(s,r) \mapsto (-s^2,r)$.

Hess (\cite{hess:1980}) extended the result by Ambrosetti and Prodi to the nonautonomous case. The Lyapunov-Schmidt decomposition in this context seems to have been established by Smiley and Chun, who also realized its potential for numerics: solving $F(u)=g$ boils down to solving the equation restricted to each fiber (\cite{smiley:1996}, \cite{smiley:1998},\cite{smiley:2000}).
In these papers, the authors are concerned with approximating the bifurcation equations, i.e., the restriction of $F(u) = g$ to the fiber $\alpha_g$ whose image contains $g$, using finite element methods. To solve the inversion problem of $F$ restricted to a given fiber, Smiley and Chun (\cite{smiley:2001}) developed a general solver for locally Lipschitz maps from $\RR^n$ to $\RR^n$, and provided examples (\cite{smiley:2001b}). Our algorithm, on the other hand, computes a point in $\alpha_g$, for arbitrary $g$ and fixed affine horizontal subspace. As a byproduct, it yields a (stable) procedure to move along $\alpha_g$.

In \cite{podolak:1976}, Podolak considered fibers for different nonlinearities. In \cite{MST:1997}, fibers were used to show that the map $G(u(t)) = u'(t) + u^3(t) - u(t)$ is a global cusp from the space of periodic functions in $\Cone([0,1])$ to $C([0,1])$: after global changes of variables, $G$ becomes $(x,y) \mapsto (x^3 - xy,y)$.

In Section \ref{sec:geo}, the consequences of flatness are presented in the general setting of a map between Banach spaces. The techniques are standard and the reader is invited to skip the section if he feels comfortable with the implications of flatness stated above. For  equation \eqref{eq:problem}, flatness follows from the coercive bound proved in Section \ref{sec:estimates}.
 The algorithm is described, first theoretically and then in more concrete terms, in Sections \ref{sec:algo} and \ref{sec:fem}. We finish with some examples in Section \ref{sec:eg}.

The authors gratefully acknowledge support from CAPES, CNPq and FAPERJ.

%

\section{Geometry of flat maps}\label{sec:geo}

Let $X$ and $Y$ be Banach spaces which split as direct sums  of \emph{horizontal} and \emph{vertical} subspaces, $X=\Wp\oplus\Vp$ and $Y=\Wd\oplus\Vd$. We assume all four subspaces to be closed and define the pairs of (bounded) complementary projections $\Pp + \Qp= I_X$ and
$\Pd + \Qd= I_X$, where $\Pp$ and $\Pd$ (resp. $\Qp$ and $\Qd$) project on horizontal (resp. vertical) subspaces.
Sets of the form $x+ \Wp$ (resp. $y + \Wd$) or $x+ \Vp$ (resp. $y + \Vd$) will be denoted by horizontal and vertical \emph{affine subspaces}.

For $v\in\Vp,$ the \emph{projected restriction}
\[ F_v:\Wp\to\Wd, \quad F_v(w)=\Pd F(w+v)\]
acts between horizontal subspaces.
A  $\Cone$ map $F: X \to Y$ is \emph{flat} with respect to a decomposition of $X$ and $Y$ as above if, for any $v \in \Vp$,  the associated projected restriction $F_v$ is a diffeomorphism.
Thus, $F$ takes horizontal affine subspaces $x+ \Wp$  injectively to their images, which are graphs of functions from $\Wd$ to $\Vd$: the surfaces $F(x+ \Wp)$ are called \emph{sheets}.

The situation is familiar: horizontal variables are trivialized by a change of variables and vertical variables are the unknowns of the bifurcation equations.
This is the content of the next proposition.

\begin{prop}\label{prop:globaldif} Let $F: X \to Y$ be flat for  decompositions of $X$ and $Y$ as above.
Then the function \[ \Phi:\tilde{X}=\Wd\oplus\Vp \to \Wp \oplus \Vp, \quad \Phi(z,v) = ((F_v)^{-1}(z),v)\]
is a $\Cone$ diffeomorphism such that $\tilde{F}=F \circ \Phi:\tilde X \to Y$ is $\tilde{F}(z,v)=(z,\phi(z,v))$ for a $\Cone$ function $\phi:\tilde X\to\Vd$.
\end{prop}
\begin{proof} As usual in Lyapunov-Schmidt arguments, consider the tentative inversion of an arbitrary point $(\wwd, \vvd) \in \Wd \oplus \Vd$,
\[ F(w+v) =  (F_v(w), \ \Qd F(w,v)) = (\wwd, \vvd).\]
By hypothesis, $w = (F_{v})^{-1}(\wwd)$. Clearly,
$\Phi= ((F_{v})^{-1}, id): \Wd \oplus \Vp \to \Wp \oplus \Vp$ and $\xi = \Phi^{-1}$
are  $\Cone$ diffeomorphisms. The rest follows from the
diagram below.
\[
\begin{array}{ccc}\scriptstyle
    (w,\ v )&\stackrel{{\scriptstyle F}}{\longmapsto}& \scriptstyle (\wwd , \ \vvd ) \\
    {\scriptstyle \xi}\searrow\nwarrow{\scriptstyle \Phi}&&\nearrow{\scriptstyle \tilde{F}=F \,\circ\, \xi^{-1}}\\
&    \scriptstyle  (\wwd ,\ v  )&
  \end{array}
\] \qed
\end{proof}

A \emph{fiber} $\alpha$ is the preimage of a vertical affine subspace. We denote by $\alpha_g$ the fiber which is the
preimage of the affine subspace $\Vd+g$. The  \emph{height} of a point $x \in X$ (resp. $y \in Y$) is the vector $\Qp x$  (resp. $\Qd y$).

\begin{prop}\label{prop:fibermanifold} Let $F: X \to Y$ be flat. Then each fiber $\alpha_g$ is a $\Cone$ surface of dimension $\dim \Vp$, which intersects each horizontal affine subspace at a unique point $x$ transversally, i.e., $X = T_x \alpha_g \oplus \Wp$. The height map $x \mapsto \Qp x$
is a diffeomorphism between the fiber $\alpha_g$ and the vertical subspace $\Vp$, with inverse  $\mathcal{H}_g: \Vp \to \alpha_g$
given by $\mathcal{H}_g(v)=v+F_v^{-1}\Pd g$.
\end{prop}

According to the proposition, $\Wp$ parametrizes (bijectively) the set of fibers, and $\Vp$ each fiber. Horizontal affine subspaces are sent injectively by $F$ to their images but fibers are not necessarily taken injectively (nor surjectively!) to vertical subspaces. In particular, the given hypotheses are not enough to imply the properness of the map $F: X \to Y$.

\begin{proof} We use the notation from the previous proposition. The change of variables $ \Phi(z,v) = ((F_v)^{-1}(z),v)$ is a diffeomorphism from each vertical affine subspace in $\tilde{X}$ to a fiber of $F$ with the property that heights are preserved. Each statement about fibers follows easily from its counterpart for vertical affine subspaces in $\tilde{X}$. \qed
\end{proof}


\begin{prop} Let $F: X \to Y$ be flat. Then sheets are manifolds which
intersect vertical affine subspaces transversally.
If $x_c$ is a critical point of $F$ contained in the fiber $\alpha$,
then $\Kernel(DF(x_c)) \subset T_{x_c}\alpha$.
\end{prop}

Transversal intersection at $F(x) \in Y$ means $Y =DF(x)\Wp \oplus \Vd$.

\begin{proof}  The change of variables $ \Phi$ is a diffeomorphism between horizontal affine subspaces, so sheets of $F$ are also the images of horizontal affine subspaces $W_Y$ under $\tilde{F}$, and hence are manifolds of codimension $\dim \Vd$. Clearly, the tangent space of a sheet at a point $F(x)$ consists of the closed vector space $DF(x)\Wp$.
To see that $DF(x)\Wp \cap \Vd = \{0\}$, suppose $DF(w + v) \tilde{w} = \tilde{v}$, for  $ x= w+v$ and $w, \tilde{w} \in \Wp, v \in \Vp, \tilde{v} \in \Vd$. Then $ \Pd DF(w + v) \tilde{w} = 0$ and since $\Pd DF(w + v) = DF_v(w)$, we have that $DF_v(w)\tilde{w} = 0$.
Now, $F_v$ is a diffeomorphism  between $v + \Wp$ and $\Wd$ and thus $\tilde{w}=0$.
Counting dimensions, we conclude that $Y =DF(x)\Wp \oplus \Vd$.

At a critical point $x_c \in \alpha$, $\Pd \circ DF(x_c): \Wp \to \Wd$ is an isomorphism by flatness,
and thus $DF(x_c): \Wp \to DF(x_c) \Wp$ also is.
Split $X = \Wp \oplus T_{x_c} \alpha$ as in Proposition \ref{prop:fibermanifold}.
An element of the kernel of $DF(x_c): \Wp \oplus T_{x_c} \alpha \to DF(x_c) \Wp \oplus \Vd$ must have null
cooordinate in $\Wp$, so that $\Kernel(DF(x_c)) \subset T_{x_c}\alpha$. \qed
\end{proof}

Thus, the projection $\Pd$ is a diffeomorphism between each sheet and $\Wd$. Fibers are disjoint, but sheets are not --- this is why some points have more than one preimage under $F$.

\section{Smoothness and Flatness}\label{sec:estimates}

Let $\Omega$ denote an open, connected, bounded set of \Rn\  with Lipschitz boundary \dOmega.
We use standard notation for Sobolev spaces $\Wk[\Omega]{k,p},\ \Wzerok[\Omega]{k,p}$; i.e., the $j$-seminorm of a function $u\in\Wk[\Omega]{k,p}$ is given by $\semnrm[u]{j,p}^{p}=\sum_{|\alpha|=j}\nrm[D^{\alpha}u]{p}^{p}$, and its $j$-norm by $\nrm[u]{j,p}^{p}=\sum_{k\le j}\semnrm[u]{k,p}^{p}$. Of special interest are the spaces with $p=2$ and $j=0,\ 1$ or 2, which we denote by $H^{0}(\Omega)=\Ltwo[\Omega],\ \Hone[\Omega]$ and $\Htwo[\Omega]$. In the case of Dirichlet boundary conditions we work mainly with the space $\Hzeroone[\Omega]$ or with $\Hzerotwo[\Omega]$. Using Poincaré's inequality we see that seminorms of the $H$ spaces are indeed norms, equivalent to the full norms, and that they are Hilbert spaces, with inner products given by
\[
\iprod[0]{u}{v}=\int_{\Omega}uv,\quad
\iprod[1]{u}{v}=\int_{\Omega}\nabla u\cdot \nabla v,\quad
\iprod[2]{u}{v}=\int_{\Omega}\lap u\cdot \lap v.
\]
We identify $\Hzeroone[\Omega]\simeq\Hnegone[\Omega]$ via
$\iprod{\tilde{u}}{\cdot}=\oneprod{u}{\cdot},$ where the tilde denotes the functional induced by an element of \Hzeroone[\Omega] and the brackets \iprod{}{}\ with no subscript
 denote the coupling between a space and its dual. Thus
\[ \negprod{\tilde{u}}{\tilde{v}}=\oneprod{u}{v},\quad
\nrm[\tilde{u}]{-1}=\semnrm[{u}]{1},
\quad
\tilde{u}_n\stackrel{\Hnegone}{\to}\tilde{u}\Leftrightarrow{u_n}\stackrel{\Hzeroone}{\to}{u}
\quad\textrm{and } \quad u\opmapsto{-\lap}\oneprod{u}{\cdot}.
\]

Recall that
 a function $f:\Omega\times \R\to\R$ is a \emph{Carathéodory} function if $s\mapsto f(x,s)$ is continuous for almost every $x$ and  $x\mapsto f(x,s)$ is measurable for every $s$.
A map of the form $u\mapsto f(\cdot,u)$ is a \emph{Nemytskii} operator.
We work with \emph{appropriate functions} $f$, which satisfy
\[
f\in \Cone[{\Omega}\times\RR],\quad \zeronrm[f(\cdot,0)]<\infty\quad\textrm{and}\quad \inftynrm[\partial_2 f]<\infty.
 \]
Set $X = H^1_0(\Omega)$, $Y = H^{-1}(\Omega)$ and,
for an appropriate $f$, the nonlinear map
\[
F: X \to Y, \quad F(u)(x)=-\lap u(x)-f(x,u(x)).
\]
The (Dirichlet) Laplacian acts weakly
and $f(\cdot,u(\cdot))$ is the functional given by  \[z\mapsto\iprod[0]{f(\cdot,u)}{z}=\int{f(x,u(x))}{z(x)}\dx.\]

Write $F(u)=-\lap u - N_f(u)$, where $N_f(u)(x)= f(x,u(x))$ is the \emph{Nemytskii operator} associated with $f$.
The estimates below are a minor extension of the properties enumerated in \cite{ambprod:book}. Throughout the text, $C$ denotes a positive constant, which may change along the argument.

\begin{prop}\label{prop:nemcone}
Let $f$ be an appropriate function. Then $F:X\to Y$ is a \Cone\ map.
\end{prop}
\begin{proof}
It suffices to show that $N_f:\Hone[\Omega]\to\Ltwo[\Omega]$ is \Cone. Indeed, this implies that $N_f:X\to Y$ is also \Cone, since $X\subset\Hone[\Omega]$ and $\Ltwo[\Omega]\subset Y$.
Notice first that $N_f:\Hone[\Omega]\to\Ltwo[\Omega]$ is well defined. Indeed, by the Taylor formula with integral remainder in $x$ there exist $a, b>0$ with
\[
\zeronrm[f(\cdot,u)]\le\zeronrm[a]+b\zeronrm[u]\le C(1+\onenrm[u]).
\]
Write the superlinear remainder
\[
e(x,h(x)) = f(x,u(x)+h(x))-f(x,u(x))-\partial_2f(x,u(x))h(x)= \delta(x,h(x))\, h(x)
\]
where
\[
\delta(x,h(x))\df\int_0^1\partial_2f(x,u(x)+\tau\,h(x))-\partial_2f(x,u(x))\ \dx[\tau].
\]
To ensure differentiability, we need to show that $\zeronrm[e]\to 0$ as $\onenrm[h]\to 0$.
By the Sobolev imbedding theorems, $h\in\Hone[\Omega]$ is also in $\Lp[\Omega]{s}$ for some $s>2$ (if $n>2$ we can take any $2<s\le2^*=\frac{2n}{n-2}$, whereas for $n\le2$ any $s>2$ works).
By Hölder's inequality, $\zeronrm[e]\le\nrm[h]{0,s}\nrm[\delta]{0,r}$, where $r>2$ is given by $\frac{1}{r}+\frac{1}{s}=\frac{1}{2}$.
Again from the imbedding theorems, $\nrm[h]{0,s}\le C\onenrm[h]$, and we are left with showing that $\nrm[\delta]{0,r}\to0$ as $\onenrm[h]\to 0$.
Switching to a subsequence if necessary, $h\to0$ pointwise a.e. so that the integrand $|\delta|^{2r}$ also converges to zero pointwise a.e., by the continuity of $f$.
Hence, by the bounded convergence theorem, $\nrm[\delta]{0,r}\to0$. This holds for any sub-subsequence and thus for $\onenrm[h]\to 0$ in general. That $z\mapsto \partial_2f(\cdot,u)z$ is a bounded map follows from $\zeronrm[\partial_2f(\cdot,u)z]\le\inftynrm[\partial_2f]\zeronrm[z]\le C\inftynrm[\partial_2f]\onenrm[z]$, completing the proof of Fréchet-differentiability.

We now show continuity of the derivative:
\[
\textrm{For any }u\in\Hone[\Omega],\quad \norm[DN_f(u+h)-DN_f(u)]\to 0\quad\textrm{whenever}\quad\onenrm[h]\to 0.
\]
Suppose $v\in\Hone[\Omega]$, set $g(x,h(x))=\partial_2f(x,u(x)+h(x))-\partial_2f(x,u(x))$ and notice that,
for $r, s$ defined above,
\begin{equation*}\label{eq:nemcone}
\nrm[g(\cdot,h)v]{0}\le\nrm[g(\cdot,h)]{0,r}\nrm[v]{0,s}\le C\nrm[g(\cdot,h)]{0,r}\nrm[v]{1}.
\end{equation*}
The constant $C$ does not depend on the function $v$ and, since $f$ is \Cone, the previous argument with the bounded convergence theorem holds, yielding
\begin{equation*}
\norm[N_f(u+h)-N_f(u)]\rightarrow0,\quad\textrm{as }\onenrm[h]\to0.
\end{equation*}
\qed
\end{proof}

In Theorem~\ref{theo:Fvdiffeo}, we exhibit direct sums $X = \Wp \oplus \Vp$ and  $Y = \Wd \oplus \Vd$ for which $F: X \to Y$ is flat. We follow closely some arguments in \cite{costa:1981}.

Denote the (possibly repeated) eigenvalues of $-\lap_D:H^2_0(\Omega) \subset H^0(\Omega) \to H^0(\Omega)$ by
$ 0 < \lambda_1 \le \lambda_2 \le \ldots$
and choose corresponding orthogonal eigenfunctions $\varphi_k, k=1,2,\ldots$.
Orthogonality holds
for the four spaces $H^2_0(\Omega)$, $H^1_0(\Omega)$, $H^0(\Omega)$
and $H^{-1}(\Omega)$.
Now let $I = [a,b]$ an interval containing $\overline{\partial_2f(\Omega,\RR)}$.
The \emph{$I$-index set} is $\calK\ = \{ i \ | \ \lambda_i \in I \}$.
The  \emph{$I$-decompositions} of $X$ and $Y$ are defined as follows.
Set $\Vp = \Vd$  be the subspace spanned by $\{\varphi_i,\, i \in \calK\}$.   Also, set $\Wp = \Vp^\perp$ and $\Wd = \Vd^\perp$. The four projections $\Pp, \Qp, \Pd$ and $\Qd$ are now orthogonal. As in the previous section,
for $v\in\Vp,$ the projected restriction $F_v:\Wp\to\Wd$
acts between horizontal subspaces.

\begin{prop} \label{prop:bound}
Let $f$ be an appropriate function, $I = [a,b] \supset \overline{\partial_2f(\Omega,\RR)}$ and
$I$-decompositions
$X = \Wp \oplus \Vp$ and $Y= \Wd \oplus \Vd$.
Then the derivatives $DF_v$ of the associated restricted projections of $F: \Wp \oplus \Vp \to \Wd \oplus \Vd$ are uniformly bounded from below.
More precisely, there exists $C > 0$ such that
\begin{equation*}
  \forall v \in \Vp \ \forall w\in \Wp \ \forall h \in \Wp,\quad\nrm[D F_v(w)h]{-1}\ge C\nrm[h]{1}.
  \end{equation*}
Also, all derivatives $DF_v$ are invertible.
\end{prop}

When $I= [a,b]$ contains only the first eigenvalue $\lambda_1$, i.e., $\calK = \{1\}$, this estimate  has been
extensively used (\cite{ambrosetti:1972}, \cite{berger:1974}). It is also used in \cite{costa:1981} in the case when $I$ contains the first $k$ eigenvalues of $-\lap_D$.  The nonautonomous case
has been considered in  \cite{hess:1980} and \cite{smiley:1996}. The result below is slightly more general.

\begin{proof}
From Proposition \ref{prop:nemcone}, each restricted projection $F_v: \Wp \to \Wd$ is $\Cone$ with derivative
$DF_v(w):\Wp \to \Wd$ given by $DF_v(w)h(x) = -\lap h(x) - \Pd \partial_2f(x,u(x))h(x)$, where $u=w+v$.
Take $h\in \Wp$ of unit norm
 and let $\gamma=(a+b)/2$.
  Adding and subtracting $\gamma h$,
  \begin{align}
    \negnrm[DF_v(w)h]&=\negnrm[\Pd(-\lap h-\gamma h)-\Pd(\partial_2f(\cdot,u)h-\gamma h)]\nonumber\\
    &\ge\negnrm[\Pd(-\lap h-\gamma h)]-\negnrm[\Pd(\partial_2f(\cdot,u)h-\gamma h)]\nonumber\\
    &\ge\negnrm[A h ]-\negnrm[B h]\label{eq:Fvest}%
    .
  \end{align}

   We bound \negnrm[B h] from above. For $z \in X$  and $w \in \Wp$,
  \begin{align*}
     \negnrm[B h]&=\stacksub{\sup\ }{\norm[z]=1}\iprod{\Pd(\partial_2f(\cdot,u) h-\gamma h)}{z}=\stacksub{\sup\ }{\norm[w]=1}\iprod{\partial_2f(\cdot,u) h-\gamma h}{w}\\
     &=\stacksub{\sup\ }{\norm[w]=1}\zeroprod{(\partial_2f(\cdot,u) -\gamma)h }{w}
     \le\inftynrm[\partial_2f(\cdot,u) -\gamma]\,\,\stacksub{\sup\ }{\norm[w]=1}\zeroprod{|h|}{|w|}.
  \end{align*}
  By Cauchy-Schwartz, the supremum is realized when $|w|$ is a scalar multiple of $|h|$, which is the case when $w=\rho h$, $\rho \in \RR$.
  Since $w$ and $h$ are unit vectors in $X$, we may take $\rho =1$ and, defining $c=\inftynrm[\partial_2f(\cdot,u) -\gamma]$,
  \begin{equation}\label{eq:Best}
     \negnrm[B h]\le c\,\zeroprod{|h|}{|h|}=c\,\zeronrm[h]^2=
     \sum_{k\not\in {\calK}}c\, h_k^2\zeronrm[\varphi_k]^2=\sum_{k\not\in {\calK}}({c}/{\lambda_k})h_k^2\onenrm[\varphi_k]^2,
  \end{equation}
  where $h_k\in\RR$ are the coefficients of the expansion of $h$ in eigenfunctions and $\calK\ = \{ \ell < \ldots < r\}$ is the $I$-index set.
  To estimate \negnrm[A h] from below, start with
  \begin{align*}
     \negnrm[A h]&=\stacksub{\sup\ }{\norm[z]=1}\iprod{\Pd(-\lap h-\gamma h)}{z}=\stacksub{\sup\ }{\norm[w]=1}\iprod{-\lap h-\gamma h}{w}\\
     &=\stacksub{\sup\ }{\norm[w]=1}\left(\oneprod{h}{w}-\gamma \zeroprod{h}{w}\right).
  \end{align*}
  Split $\Wp=W_-\oplus W_+$, where
  \[ W_-=\{u\st u=\sum_{k < l}u_k\varphi_k\},\quad W_+=\{u\st u=\sum_{k > r}u_k\varphi_k\}\]
  are orthogonal subspaces in the $H^1$ and $H^0$ norms. Split $h = h_- + h_+$ and set $w = h_+-h_-\in X$, clearly a unit vector:
  \begin{align*}
     \negnrm[A h]&\ge\oneprod{h}{h_+-h_-}-\gamma \zeroprod{h}{h_+-h_-}
     =(\onesemnrm[h_+]^2-\gamma\zeronrm[h_+]^2)+(\gamma\zeronrm[h_-]^2-\onesemnrm[h_-]^2)\\
     &=\sum_{k> r}h_k^2(\onesemnrm[\varphi_k]^2-\gamma\zeronrm[\varphi_k]^2)+%
     \sum_{k< l}h_k^2(\gamma\zeronrm[\varphi_k]^2-\onesemnrm[\varphi_k]^2)\\
     &=\sum_{k> r}(1-{\gamma}/{\lambda_k})h_k^2\onesemnrm[\varphi_k]^2+%
     \sum_{k< l}({\gamma}/{\lambda_k}-1)h_k^2\onesemnrm[\varphi_k]^2.
  \end{align*}
  Notice that $(1-{\gamma}/{\lambda_k})$ is positive (resp. negative) for $k>r$ (resp. $k<\ell$).
  Then
  \begin{equation}\label{eq:Aest}
    \negnrm[A h]\ge\sum_{k\not\in {\calK}}|1-{\gamma}/{\lambda_k}|\,h_k^2\onesemnrm[\varphi_k]^2=\sum_{k\not\in {\calK}}(C_k/\lambda_k)h_k^2\onesemnrm[\varphi_k]^2,
  \end{equation}
   where $C_k=|\lambda_k-\gamma|$.
  Combining equations \eqref{eq:Fvest}, \eqref{eq:Best} and \eqref{eq:Aest},
  \begin{align*}
    \negnrm[DF_v(w)h]&\ge\sum_{k\not\in {\calK}}(C_k-c)/\lambda_k\,h_k^2\onesemnrm[\varphi_k]^2
    \ge\left(\stacksub{\inf\ }{k\not\in {\calK}}(C_k-c)/\lambda_k\right)\sum_{k\not\in {\calK}}h_k^2\onesemnrm[\varphi_k]^2\\
    &=\left(\stacksub{\inf\ }{k\not\in {\calK}}(C_k-c)/\lambda_k\right)\onesemnrm[h]^2
    =C\onesemnrm[h]^2=C.           
  \end{align*}
  The infimum above is achieved at one of the outer eigenvalues closest to $[a,b]$,
  proving the injectivity of $DF_v(w)$. We now show that $DF_v(w)$
  is a Fredholm operator of index zero, and hence surjective.

  Indeed, $-\lap: X \to Y$ is an isomorphism and $DF(u):X \to Y$ given by $DF(u)z = -\lap z - \partial_2f(\cdot,u)z$
  is obtained by adding a compact operator, from Proposition \ref{prop:nemcone}. Now, $ DF_v(w) = \Pd \circ DF(v+w) \circ \iota$,
  where the projection $\Pd: Y \to \Wd$ and the inclusion $\iota: \Wp \to X$ are Fredholm operators, whose indices
  add to zero. Thus $ DF_v(w)$ is also Fredholm of index zero. \qed
  \end{proof}

Recall Hadamard's global inversion theorem (\cite{berger:book}).

\begin{lem}\label{lem:hadamard}
  Let $\Phi:X \to Y$ be a \Cone\ map between Banach spaces $X$ and $Y$ such that $D\Phi(u)$ is invertible for each $u \in X$.
   Suppose there exists $C >0$ such that
  \begin{equation*}
  \forall u,h \in X\quad\nrm[D\Phi(u)h]{}\ge C\nrm[h]{}.
  \end{equation*}
  Then $\Phi$ is a global \Cone-diffeomorphism.
\end{lem}

\begin{theo}\label{theo:Fvdiffeo}
Let $f$ be an appropriate function, $I = [a,b] \supset \overline{\partial_2f(\Omega,\RR)}$. Then the map $F: X \to Y$ is flat with respect to the
$I$-decompositions
$X = \Wp \oplus \Vp$ and $Y= \Wd \oplus \Vd$.
\end{theo}
\begin{proof}
Simply combine the proposition and lemma above. \qed
\end{proof}

There is an analogous statement for $\tilde{F}: H^2_0 \to H^0$.

From the previous section, since $F$ is flat, its domain is foliated by $\Cone$ fibers of dimension $\dim \Vp = \dim \Vd$, which are transversal to the horizontal affine subspaces $x + \Wp$ and are parameterized diffeomorphically by the height function. The bound in Proposition \ref{prop:bound} allows to make precise the idea that fibers are uniformly steep and sheets are uniformly flat.

\begin{prop} \label{prop:stability}
   Let $f$ be appropriate, $I = [a,b] \supset \overline{\partial_2f(\Omega,\RR)}$ and $\Wp \oplus \Vp$ and $\Wd \oplus \Vd$ be the corresponding $I$-decompositions of $X$ and $Y$. Denote by \calK\ an index set for the basis of eigenfunctions spanning $\Vp= \Vd$ and let
   \[
    u(t)=w(t)+ v(t), \hbox{ for }  w(t)\in\Wp,\ v(t)=\sum_{i\in\calK}t_i\,\varphi_i\in\Vp, \ t = (t_1, \ldots,t_{|\calK|})
   \]
   be a parametrization of a fiber $\alpha$ of the flat map  $F: X \to Y$.
   Then there exists a constant $C$, independent of $t$, such that
   \[
    \onenrm[\nabla_{t} w(t)]\le C \sum_{i\in {\calK}} \onenrm[\varphi_i].
   \]
   In particular, there exist constants $A,\ B$, independent of $t$, such that
   \[ \onenrm[w(t)]\le A+B\norm[t].\]
   Let $u \in X$ and consider the sheet ${\cal W}_u =F(u +\Wp)$ with tangent space $T_{F(u)}{\cal W}_u$ at $F(u)$.
   Then the angle between a vector in $T_{F(u)}{\cal W}_u$ and its orthogonal projection in $\Wd$ is (uniformly) bounded above by a constant less then $\pi/2$.
\end{prop}

This result is a source of robustness for the numerics in the next sections.

\begin{proof}
  Fibers are inverses under $F$ of vertical affine subspaces in $Y$. Taking derivatives of $\Pd F(u(t))=\textrm{const.}$ with respect to $t_i$,
\begin{equation*}
 D\Pd F(u(t))\ \partial_{t_i}u(t)\ =\ \Pd DF(u(t))\ \partial_{t_i}u(t)=\Pd DF(u(t))\ ( \partial_{t_i}w(t)+\varphi_i)=0.
\end{equation*}
Since for any $h\in \Wp$ we have $\Pd DF(u(t))h=DF_{v(t)}(w(t))h$, for $h = \partial_{t_i}w(t)$,
    \[
    DF_{v(t)}(w(t))\partial_{t_i}w(t)=\Pd DF(u(t))\partial_{t_i}w(t)=-\Pd DF(u(t))\varphi_i.
    \]
Using  first the lower bound in Proposition \ref{prop:bound} and then the boundedness of $DF$,
\[
    C_1 \onenrm[\partial_{t_i}w(t)]\le \negnrm[DF_{v(t)}(w(t))\partial_{t_i}w(t)]=\negnrm[\Pd DF(u(t))\varphi_i]\le C_2 \onenrm[\varphi_i],
\]
for some  constant $C_2$. Thus $\onenrm[\nabla_{t} w(t)]\le C \sum_{i\in {\calK}} \onenrm[\varphi_i]$, for some other  constant $C$.
A bound of the form $\onenrm[w(t)]\le A+B\norm[t]$ is now immediate.

To see that $T_{F(u)}{\cal W}_u$ is bounded away from the vertical subspace, consider the sequence of simple estimates, for $h\in\Wp$:
\[ C_1 \onenrm[h] \le \negnrm[\Pd DF(u)h] \le \negnrm[DF(u)h] \le C_3 \onenrm[h]\]
The cosine between a vector $ DF(u)h \in T_{F(u)}{\cal W}_u$ and the horizontal subspace $\Wd$ is
$\negnrm[\Pd DF(u)h]/\,\negnrm[DF(u)h]$, which is bounded from below by $C_1/C_3$.
\qed
\end{proof}

A regularity theorem for $F$, in the sense that $g\in\Cinfty[\Omega]\Rightarrow u\in\Cinfty[\Omega]$ for $F(u)=g$, would imply that points in the same fiber have the same  differentiability.

\section{Finding Preimages under $F$}\label{sec:algo}

We now describe an algorithm to solve
 $F(u)= -\lap u - f(\cdot,u) = g, \ u|_{\partial \Omega} = 0.$
 The details of implementation are handled in Section \ref{sec:eg}.
The equation is interpreted as the computation of the preimages of $g$ under $F: X = H^1_0(\Omega) \to Y = H^{-1}(\Omega)$.

For $g \in H^0$, there is an alternative point of view, which we do not treat in this paper: one might work instead with  $\tilde{F}:H^2_0(\Omega) \to H^0(\Omega)$, which shares the same geometric properties than $F$, as commented below Theorem~\ref{theo:Fvdiffeo}.
However, the discretizations will be performed by choosing appropriate finite elements, and the programming becomes easier
for the less restrictive basis used in $H^1$, as opposed to the finer elements in $H^2$. Clearly, the preimages of $g \in H^0$ under $F$ are in $H^2$.

We assume that the nonlinearity $f$ is an appropriate function and the interval $I=[a,b]$ contains $\overline{\partial_2f(\Omega,\RR)}$,
 so that, by Theorem \ref{theo:Fvdiffeo}, $F: X \to Y$ is flat with respect to the $I$-decompositions $X=\Wp \oplus \Vp$ and $Y=\Wd \oplus \Vd$.

In a nutshell, split $g= \Pd g + \Qd g = g_W + g_V$.
The inversion under $F$ of the vertical affine space $g_W + \Vd$ gives rise to a fiber $\alpha_g$ which contains all the solutions
of the original equation. The algorithm first identifies, for a fixed $v\in\Vp$, a point $u_g \in \alpha_g \cap \{v+\Wp\}$:
 this is essentially handling the equation $\Pd F(u_g) = g_W$ in $\{v+\Wp\}$. The search for solutions then boils down to a
finite dimensional problem along $\alpha_g$, which corresponds to the bifurcation equation $\Qd F(u) = g_V $.

\subsection{Moving in the space of fibers}\label{subsec:movinfib}

Our first goal is to reach a point $u_g$ in the fiber $\alpha_g=F^{-1}(g+\Vd)$, or more realistically,
close to it. For an arbitrary $v \in \Vp$, we search
the unique point $u$ of $\alpha_g$
in the horizontal affine space $v + \Wp$ given by Proposition \ref{prop:fibermanifold}. This is equivalent to solving
\begin{equation*}
\Pd F(v+w) = F_v(w)=\Pd g, \quad  w \in \Wp.
\end{equation*}
Since $F$ is flat, for each $v \in \Vp$, $F_v :  \Wp \to \Wd$ is a diffeomorphism so that, for any $w \in \Wp$,
$DF_v(w) = -\Pd\lap -\Pd \partial_2 f(\cdot,u)$
 is an isomorphism. Thus, we may consider
Newton's method on $F_v$ to move horizontally in $\Wp$.
However, if we restrict our computations to horizontal subspaces, our finite elements discretizations will not yield sparse matrices.
It is natural, then, to search for an extension to the full space of the operator $DF_v$ which is invertible and easy to compute.

For $u\in X$ we define the linear operator $L_c(u): X \to Y$ by
  \[
  L_c(u) z=-\lap z- \Pd \partial_2f(\cdot,u) \Pp z - c\, \Qd \Qp z =
  -\lap z- \Pd \partial_2f(\cdot,u) \Pp z - c\, \Qp z,
  \]
since $\Vp = \Vd$.
\begin{prop}\label{prop:Lcu} Write $u = w + v \in \Wp\oplus \Vp$. The restrictions of $L_c(u): X \to Y$ to $\Wp$ and $\Vp$ are $DF_{v}(w): \Wp \to \Wd$ and $-\lap -c\,I:\Vp \to \Vd$.
\end{prop}
\begin{proof}
 For $z = z_W+z_V \in \Wp\oplus \Vp$,
\begin{align*}
  L_c(u)z&=-\lap (z_W+z_V)- \Pd \partial_2f(\cdot,u) \Pp (z_W+z_V)-c\,\Qp(z_W+z_V)\\
  &= \left(-\lap z_W - \Pd \partial_2f(\cdot,u) z_W\right)+\left(-\lap z_V -c\,z_V\right)\\
  & = DF_{v}(w)z_W+( -\lap-c\,I) z_V.
\end{align*}
    \qed
\end{proof}

Notice that $L_c(u)$ is an integro-differential operator. This is not a problem for the finite elements discretization and has an added bonus the preservation of sparsity of the relevant matrices.

To reach $u_g \in v + \Wp$ in the fiber $\alpha_g$, start with an arbitrary $u_0 \in v+\Wp$.
To update $u_n\in v+\Wp$, solve
\[ L_c(u_{n}) h = \Pd  (g-F(u_{n}))
\quad \textrm{and set}\quad u_{n+1}=u_{n}+\Pp \, h. \]
The projection in the formula for $u_{n+1}$ is redundant, but it removes possible numerical errors
that might give rise to a nontrivial vertical component when solving for $h$.
Actually, from Proposition \ref{prop:Lcu},
\begin{equation}\label{eq:Lc}
u_{n+1}=u_{n}+\Pp \, \tilde{h}, \quad\textrm{where }\quad L_c(u_{n}) \tilde{h} = g-F(u_{n}).
\end{equation}

Numerical errors self-correct, in the spirit of Newton's method: termination occurs once the norm of the error
$e_n=\Pd  (g-F(u_n))$ is sufficiently small, yielding a point $u_g$ essentially in $\alpha_g$.

In principle, convergence is not expected and might require prudence: inversion of points along the horizontal segment joining $\Pd F(u_0)$ to $\Pd g$. This always works in exact arithmetic, since $F_{v_0}$ is a \Cone\  diffeomorphism.

The algorithm above also implements the diffeomorphism $\mathcal{H}_g: \Vp \to \alpha_g$
introduced in  Proposition \ref{prop:fibermanifold} ---
it suffices to start from $v \in \Vp$ and
move horizontally until Newton's iteration reaches $\alpha_g$.

\subsection{Moving Along a Fiber}

Once $u_g \in \alpha_g$ is identified, the original problem reduces to a finite-dimensional issue. Said differently, we should invert the restriction of $F$ to $\alpha_g$, which amounts to inverting
$\mathcal{F}_g : \Vp \to \Vd$ given by $\mathcal{F}_g = \Qd \circ F \circ \mathcal{H}_g$.

Actually, when a value $\mathcal{H}_g(v_0)$ has been computed, we may compute $\mathcal{H}_g(v_0 + p)$, for $p \in \Vp$,
by starting from $\mathcal{H}_g(v_0)+p$ instead of $v_0+p$ and then moving horizontally with Newton's method until we reach $\alpha_g$.
The advantage lies in the fact that, for small $p$, there will be less horizontal displacement with the new initial condition.
The resulting algorithm is essentially a predictor-corrector scheme.
Figure \ref{fig:mapfiber} illustrates the procedure in the case when $\dim \Vp = |\calK| = 1$, so that $\alpha_g$ is a curve.

\begin{figure}[h]
\centering
\mbox{\subfigure{\includegraphics[width=.45\linewidth]{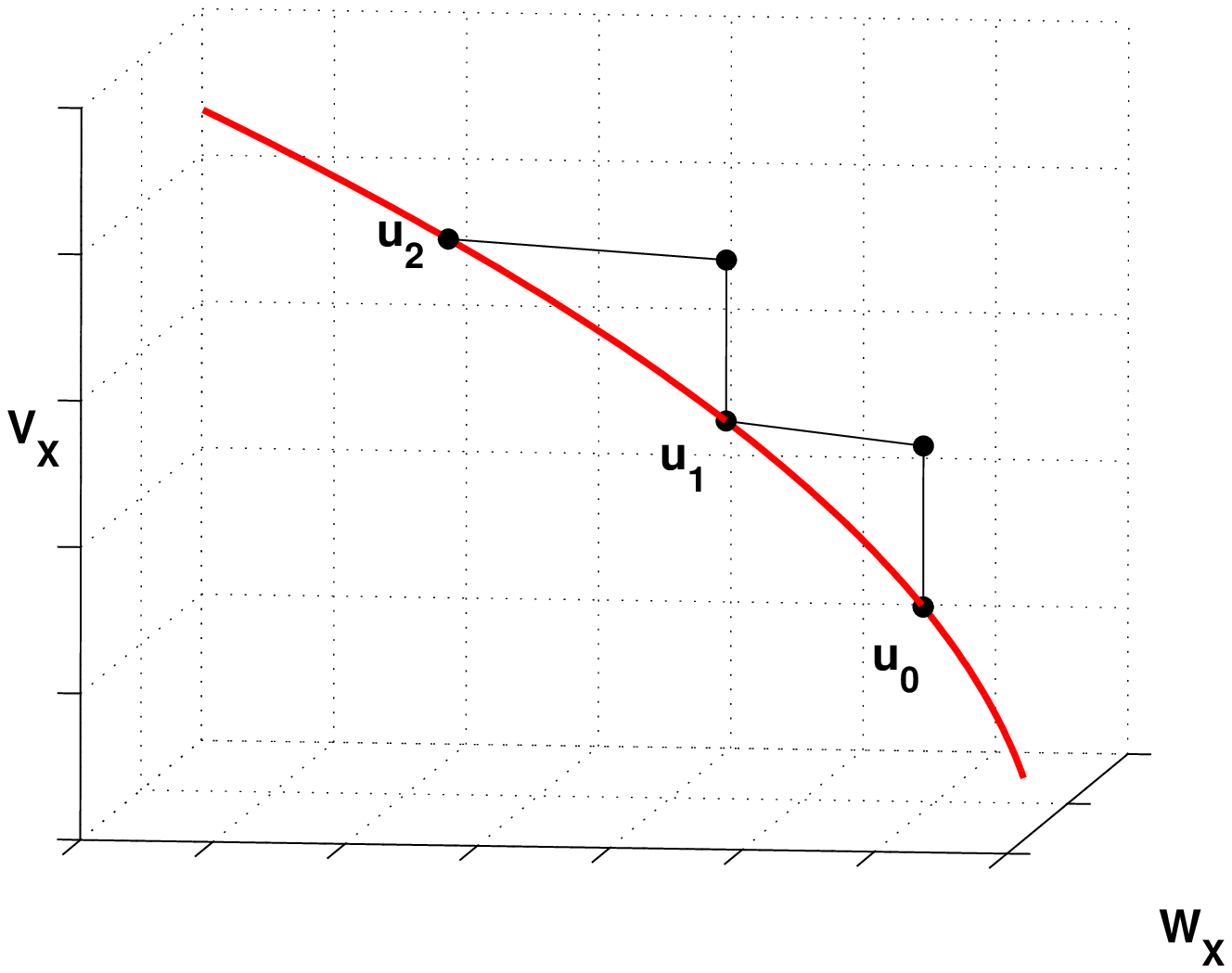}}\quad\quad
\subfigure{\includegraphics[width=.45\linewidth]{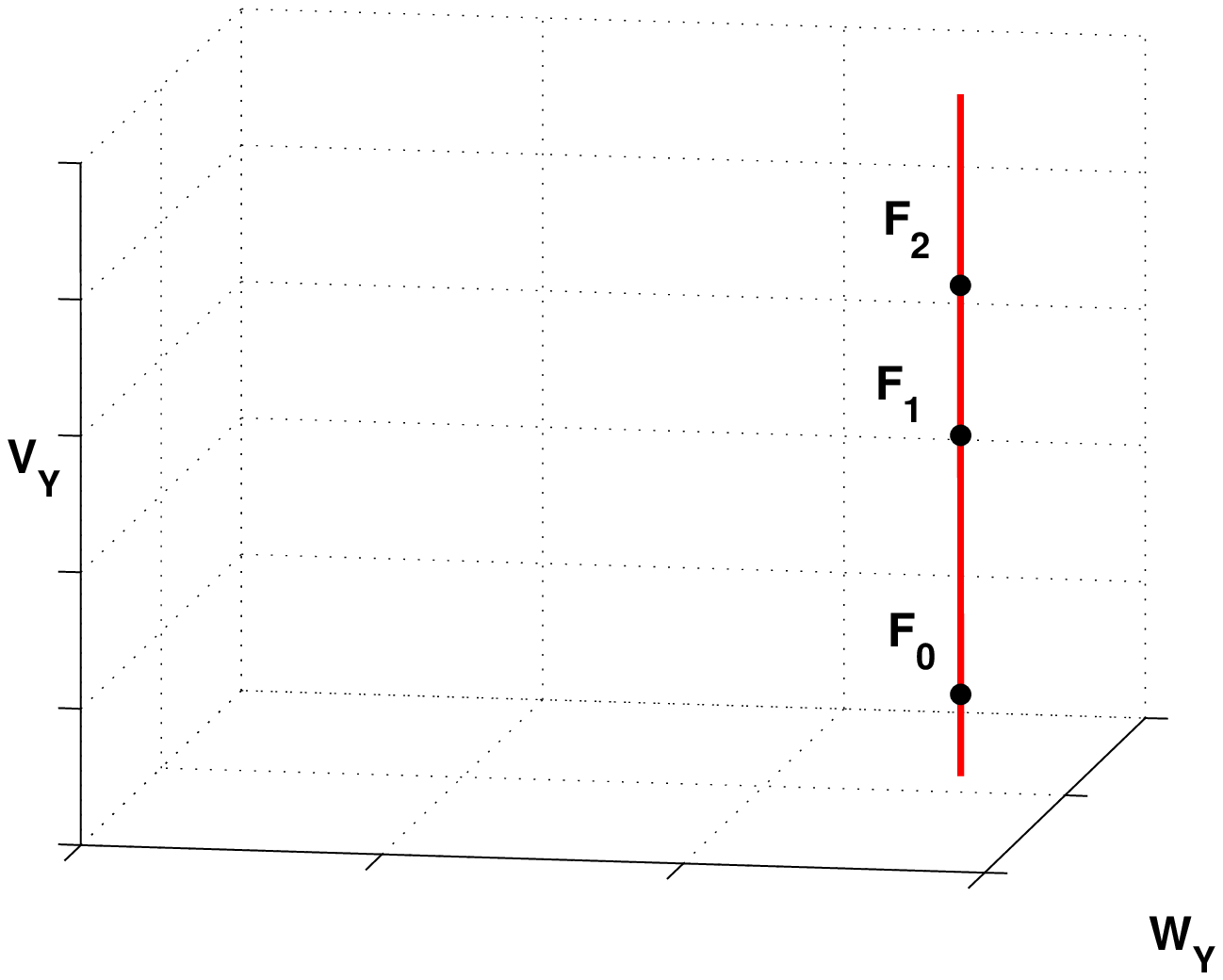} }}
\caption{Mapping a  1-D fiber} \label{fig:mapfiber}
\end{figure}

\section{ Implementing the algorithm}\label{sec:fem}

We now describe the finite element discretization of the algorithm above.
Given a domain $\Omega \subset \RR^n$, we consider a triangulation with interior vertices $\nu_j,\ j=1,\ldots,N$. The  \emph{nodal functions} $\uh[\psi]_j$
are continuous functions which are linear on each element, with values on vertices given by $\uh[\psi]_j(\nu_k) = \delta_{jk}$. The nodal functions span the finite element space \fe.

\subsection{Moving Horizontally}

For $\uh \in \fe \subset X = \Hzeroone[\Omega]$ and $g \in \Ltwo[\Omega] \subset Y = \Hnegone[\Omega]$,
we now discretize
\begin{equation}\label{eq:todisc}
L_c(u)z=-\lap z - \Pd \partial_2f(\cdot,u) \Pp z  - c\, \Qp z= g-F(u).
\end{equation}
As described in Section~\ref{subsec:movinfib}, this is the main step to identify the fiber $\alpha_g$.

Functions in $X$ and $Y$ are approximated by elements in \fe, but their identification is different.
We take the nodal functions $\{\uh[\psi]_j\}$ as a basis for $\Xh = \fe \subset X$. For
$\uh \in \Xh$,  we have $\uh(x)=\sum_j\unbar_j\uh[\psi]_j(x)$, where $\unbar_j = \uh(\nu_j)$.

For functions $g \in \fe \subset Y$,
we are interested in the values of the (independent) functionals $\ell_i(g)=\zeroprod{\uh[\psi]_i}{g}=\uhat[g]_i$.
We take in $\fe$
the dual basis $\ell_j^*, j = 1,\ldots,N$, defined by $\ell_i(\ell_j^*)=\delta_{ij}$,  so that
$g(x) = \sum_j \uhat[g]_j \ \ell_j^*(x)$. The \emph{mass matrix} $\mat$ changes coordinates:
\[
\mat\,\unbar=\uhat,\quad \mat_{ij}=\zeroprod{\uh[\psi]_i}{\uh[\psi]_j}.
\]

In coordinates $\unbar$ and $\uhat[g]$, the expression $-\lap u = g$ becomes
\[
\mat[K]\,\unbar=\uhat[g],\quad \mat[K]_{ij}=\oneprod{\uh[\psi]_i}{\uh[\psi]_j},
\]
where $\mat[K]$ is the standard \emph{stiffness matrix}.

Define inner products
$\iprod[\Xh]{\unbar_1}{\unbar_2} =\iprod{\mat[K]\unbar_1}{\unbar_2}$
and
$\iprod[\Yh]{\uhat[g]_1}{\uhat[g]_2} =\iprod{\mat[K]^{-1}\uhat[g]_1}{\uhat[g]_2}$ in $\Xh$ and
$\Yh$
(here $\iprod{\,}{\,}$ denotes the standard inner product in Euclidean space). Notice the isometry
$\iprod[\Xh]{\unbar_1}{\unbar_2}=\iprod[\Yh]{\uhat[g]_1}{\uhat[g]_2}$, where $\mat[K]\unbar_i=\uhat[g]_i,\ i=1,2.$
The eigenpairs $\lambda_i$, $\varphi_{i},\ i\in {\calK}$,  have approximations $\lambda^h_i$, $\uh[\varphi]_i\in \Xh$ obtained by solving
\[
\mat[K] \uh[\underline{\varphi}]_i = \lambda^h_i \mat[M] \uh[\underline{\varphi}]_i,\quad \iprod[\Xh]{\uh[\underline{\varphi}]_i}{\uh[\underline{\varphi}]_i}=1.
\]
The approximate eigenfunctions $\uh[\varphi]_i$ span $\Vph =\Vdh\subset\Xh,$
which may be taken arbitrarily close to the vertical subspaces associated to the index set $\calK$ by choosing a small value of $h$. Similarly, the horizontal subspaces $\Wp$ and $\Wd$ are approximated by $\Wpt$ and $\Wdt$, the orthogonal complements of $\Vph$ and $\Vdh$ in $X$ and $Y$ respectively.

Proposition \ref{prop:stability} implies a certain kind of stability. The uniform steepness of fibers and the uniform flatness of sheets ensure preservation of flatness. More precisely,
if $F:X = \Wp \oplus \Vp \to Y= \Wd \oplus \Vd$ is flat,
then it is also flat with respect to the decompositions
\[X = \Wpt \oplus \Vph, \quad  Y =  \Wdt \oplus \Vdh, \]
provided that the respective subspaces are sufficiently close to each other.

Define $\Fh: \Xh \to \Yh$ by
\[
\Fh(\unbar)=\mat[K]\,\unbar-\mat\,f(\unbar)=\uhat[F],
\]
where $f(\unbar)$ is the vector whose coordinates are $f(\nu_j,\unbar_j)$. For small $h$, $\Fh$ is flat with
respect to the decompositions $\Xh = \Wph \oplus \Vph$
and $\Yh =  \Wdh \oplus \Vdh$, where the horizontal spaces $\Wph$ and $\Wdh$ are orthogonal to $\Vph = \Vdh$
in the discrete inner products.

Assuming $z\in\Xh$ in equation \eqref{eq:todisc} and taking the $L^2$ inner product of with the nodal function $\uh[\psi]_i \in X$, we obtain
\[
(\mat[K]\unbar[z])_i - \zeroprod{\Pd \partial_2f(\cdot,u) \Pp z}{\uh[\psi]_i}-c\zeroprod{\Qp z}{\uh[\psi]_i}=\uhat[g]_i-\uhat[F]_i.
\]
Since $\partial_2f(\cdot,u) \Pp z\in\Ltwo[\Omega]$, $\zeroprod{\Pd \partial_2f(\cdot,u) \Pp z}{\uh[\psi]_i}=\zeroprod{\partial_2f(\cdot,u) \Pp z}{\Pp\uh[\psi]_i}$ and we are left with discretizing $\Pp=I-\Qp.$
In coordinates, $\displaystyle
\Qph \unbar[z]=\sum_{k}\iprod[\Xh]{\unbar[z]}{\uh[\underline{\varphi}]_k}
\,\uh[\underline{\varphi}]_k.
$

Once we write $z=\sum_j\unbar[z_j]\uh[\psi]_j$, the discretization $\mat[L^h]$ of $L_c(u)$ is expressed in terms of the inner products
\[
\oneprod{\uh[\psi]_j}{\uh[\varphi]_i},\
\zeroprod{\partial_2f(\cdot,u)\uh[\psi]_j}{\uh[\psi]_k},\
\zeroprod{\partial_2f(\cdot,u)\uh[\psi]_j}{\uh[\varphi]_i},\
\zeroprod{\partial_2f(\cdot,u)\uh[\varphi]_i}{\uh[\varphi]_{i'}},
\]
where $j,k=1,\cdots,N$ and $i,i'\in {\calK}$. In our computations, we replaced $\partial_2f(\cdot,u)$ by the vector with coordinates $f(\nu_j,\unbar_j)$.

The discretization of the the uptdating $u_n \mapsto u_{n+1}$ defined in \eqref{eq:Lc} becomes then
\[
\unbar:=\unbar+\Pph\,\underline{\eta}, \quad\textrm{where}\quad \mat[L^h]\,\underline{\eta}=\hat{g}-\Fh(\unbar).
\]

\subsection{Moving along a fiber}

The finite dimensional inversion of a computable function, as $\mathcal{F}_g = \Qd \circ F \circ \mathcal{H}_g$,
is not a trivial issue.
What is needed is a solver which  takes into account the special features of maps between vertical subspaces (or, more geometrically, from fibers to vertical subspaces). In the
examples below, except for the last one, $|\calK| = 1$.

The fact that there was  a finite dimensional reduction for the equation $F(u) = g$ was implicit in \cite{berger:1974}, restated in \cite{smiley:1996} and stated in a very explicit form (Theorem 2.1) in \cite{smiley:2000}. Smiley and Chun (\cite{smiley:2001b}) considered the numerical inversion of restrictions of $F$ to given fibers, using an inversion algorithm they developed for locally Lipschitz maps between Euclidean spaces (\cite{smiley:2001}).

As usual, the more we know about $F$, the sturdier the numerics. The Ambrosetti-Prodi case is rather simple: the nonlinearity $f$ interacts only with $\lambda_1$ and $\partial_2f' >0$. The map $F$ sends fibers to folded vertical lines: as the height of a point in the fiber goes from $-\infty$ to $\infty$, the height of its image goes monotonically from $-\infty$ to a maximal point and
then decreases monotonically to $-\infty$. Dropping convexity allows for loss of monotonicity, but not of asymptotic behavior, as we shall see in the examples of the next section.

There are  theoretical results (\cite{podolak:1976}) that guarantee that under different, but stringent, hypotheses the Ambrosetti-Prodi pattern along fibers carries through. The numerics may be performed in more general conditions, providing strong evidence to the eventual outcome.

\section{Numerical Examples}\label{sec:eg}

All the examples in this chapter relate to the autonomous equation
\begin{equation*}
F(u) = - \lap u - f(u) = g
\end{equation*}
with Dirichlet boundary conditions on the
rectangle $\Omega=[0,1]\times[0,2]$, for which the smallest three (simple) eigenvalues are
\[ \lambda_1 = \frac{5}{4} \pi^2 \approx 12.34 ,\quad \lambda_2 = 2 \pi^2 \approx 19.74,
\quad \lambda_3 = \frac{13}{4} \pi^2 \approx 32.07.\]
The nonlinearities $f$ are always appropriate functions. When $f$ is convex, we take $f'(x)= \alpha \arctan (x) + \beta$ for different choices of the asymptotic parameters $\alpha$ and $\beta$.

Recall that first, given a horizontal affine subspace $v+ \Wp$ and a right hand side $g \in Y$, the algorithm searches for a point $u_g \in v+ \Wp$ in the fiber $\alpha_g$, using the iteration described in Section \ref{subsec:movinfib}. In each step we solve Equation \eqref{eq:Lc}: in the examples below, $c=0$.
Then inversion of $F:\alpha_g \to g + \Vd$ with basepoint $u_g$ obtains, in principle, all solutions of the equation.

The triangulation was generated with Matlab's PDE Toolbox and the matrices were programmed from scratch and compared to those computed by the toolbox, whenever possible.

\subsection{Finding $u_g$ in $\alpha_g$} \label{sec:movhor}

Consider the Ambrosetti-Prodi situation with $ f'(x)= \alpha \arctan (x) + \beta$
satisfying
$ \lim_{x \to \pm \infty} f'(x)= \lambda_1 \pm (\lambda_2 - \lambda_1)/2$.
The right-hand side $g(x) = - 100 x(x-1)y(y-2)$
is chosen to resemble a very negative multiple of $\varphi_1$.

Usually one or two iterations of the horizontal step lead to an error which can only decrease by choosing a finer triangulation.
Newton's iteration was very successful: continuation arguments were not necessary.
An $m$-triangulation ${\cal T}_m$ splits each interval $[0,1]$ and $[0,2]$ in $2^m$ equal subintervals.
For $u_0 = 100 \varphi_2$, we present the normalized horizontal errors $e_n=|| \Pd (g - F(u_n))|| /|| \Pd (g - F(u_0))||$, $n=1,2$ and $3$, for triangulations with $m=3,4$ and $5$  for the $H^{-1}$ and $H^0$ norms.
\begin{center}
\begin{tabular}{|l|r|r|r|}
  \hline
  m & $e_1\, (H^{-1})$ & $e_2$& $e_3$\\
  \hline
  3 & 1.42E-2 & 5.27E-5 & 4.48E-8  \\
  \hline
  4 & 1.70E-2 & 1.12E-4 & 3.93E-8  \\
  \hline
  5 & 1.75E-2 & 1.31E-4 & 4.25E-8 \\

  \hline
\end{tabular}, \quad
\begin{tabular}{|l|r|r|r|}
  \hline
  m & $e_1\, (H^{0})$ & $e_2$& $e_3$\\
  \hline
  3 & 1.97E-2 & 9.37E-5 & 7.45E-8  \\
  \hline
  4 & 2.36E-2 & 1.74E-4 & 1.21E-7  \\
  \hline
  5 & 2.44E-2 & 1.93E-4 & 1.11E-7 \\

  \hline
\end{tabular}
\end{center}

In Figure \ref{fig:movhor} we show $g$ and the function $u_3$.

\begin{figure}[h]
\centering
\mbox{\subfigure{\includegraphics[width=.45\linewidth]{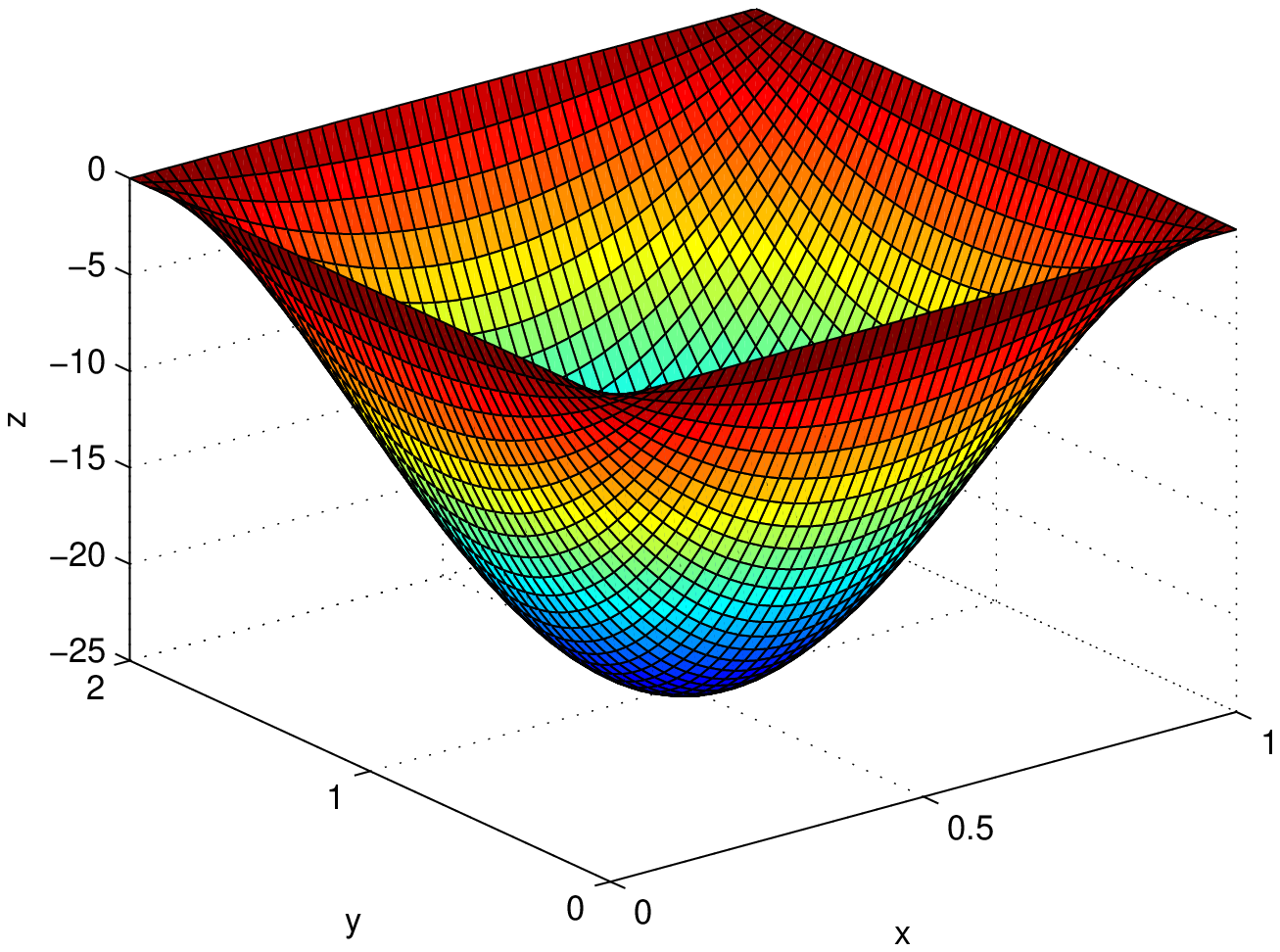}}\quad\quad
\subfigure{\includegraphics[width=.45\linewidth]{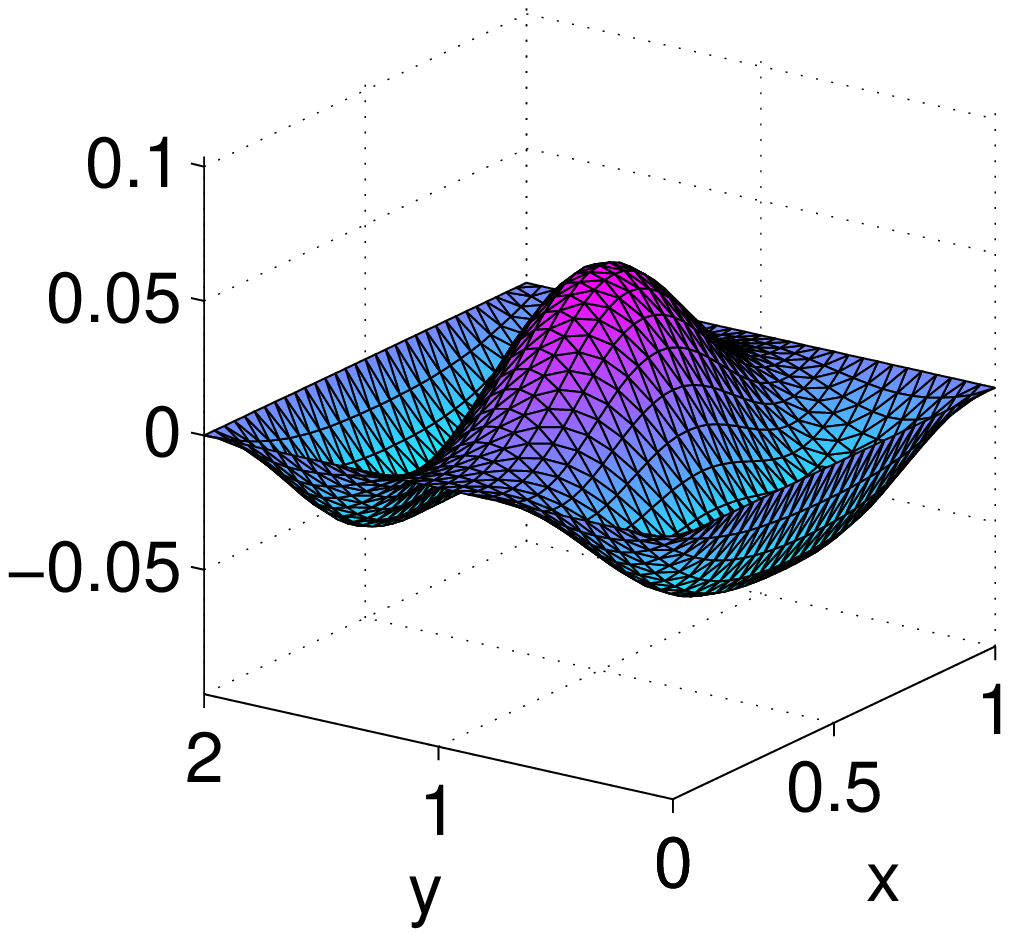} }}
\caption{A right-hand side $g$  and $u_g \in \alpha_g$ obtained from $u_0 \equiv 0$.} \label{fig:movhor}
\end{figure}

\subsection{Finding Solutions: Moving Along a Fiber} \label{sec:movver}

In this section, we prescribe a nonlinearity $f$, a point $u_0$ and study the restriction of $F$ to the fiber $\alpha_g$ for $g=F(u_0)$.
The eigenfunctions $\varphi_k$ are normalized in the \Hone-norm (resp. \Hnegone) in the domain (resp. counter-domain).

Each example starts with two graphs. In the first, we
plot $f'$ and mark with dotted lines the relevant eigenvalues.
The second graph plots the height of $F(u)$ against the height of a point $u \in \alpha_g$. Informally, it shows how the image of a fiber goes up and down: in particular, it indicates the number of solutions of
$F(u) = F(u_0)$. Additional solutions to the equation are then presented.
\subsubsection{Dolph-Hammerstein}

\begin{figure}[h]
\centering
\mbox{\subfigure{\includegraphics[width=.45\linewidth]{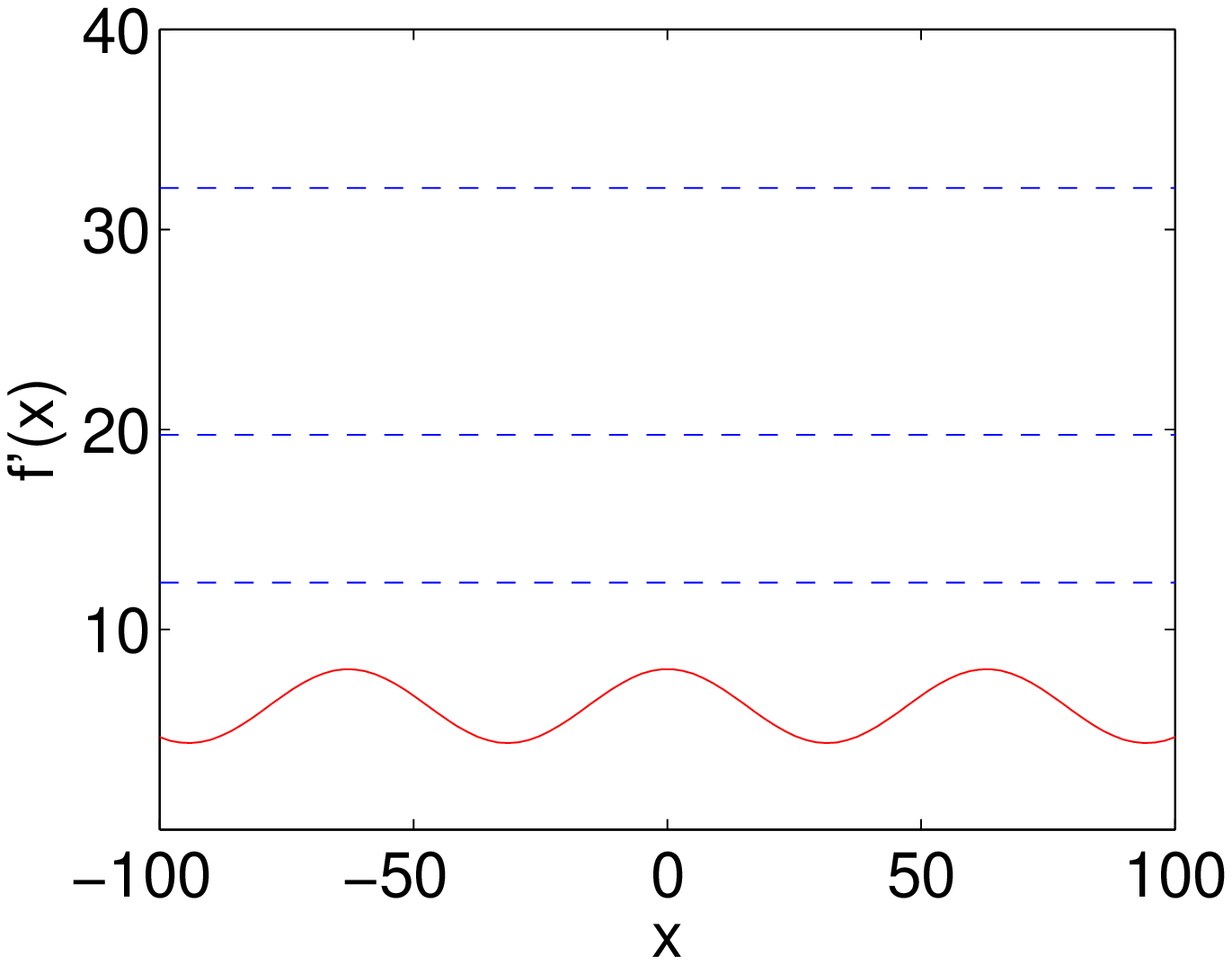}}\quad\quad
\subfigure{\includegraphics[width=.45\linewidth]{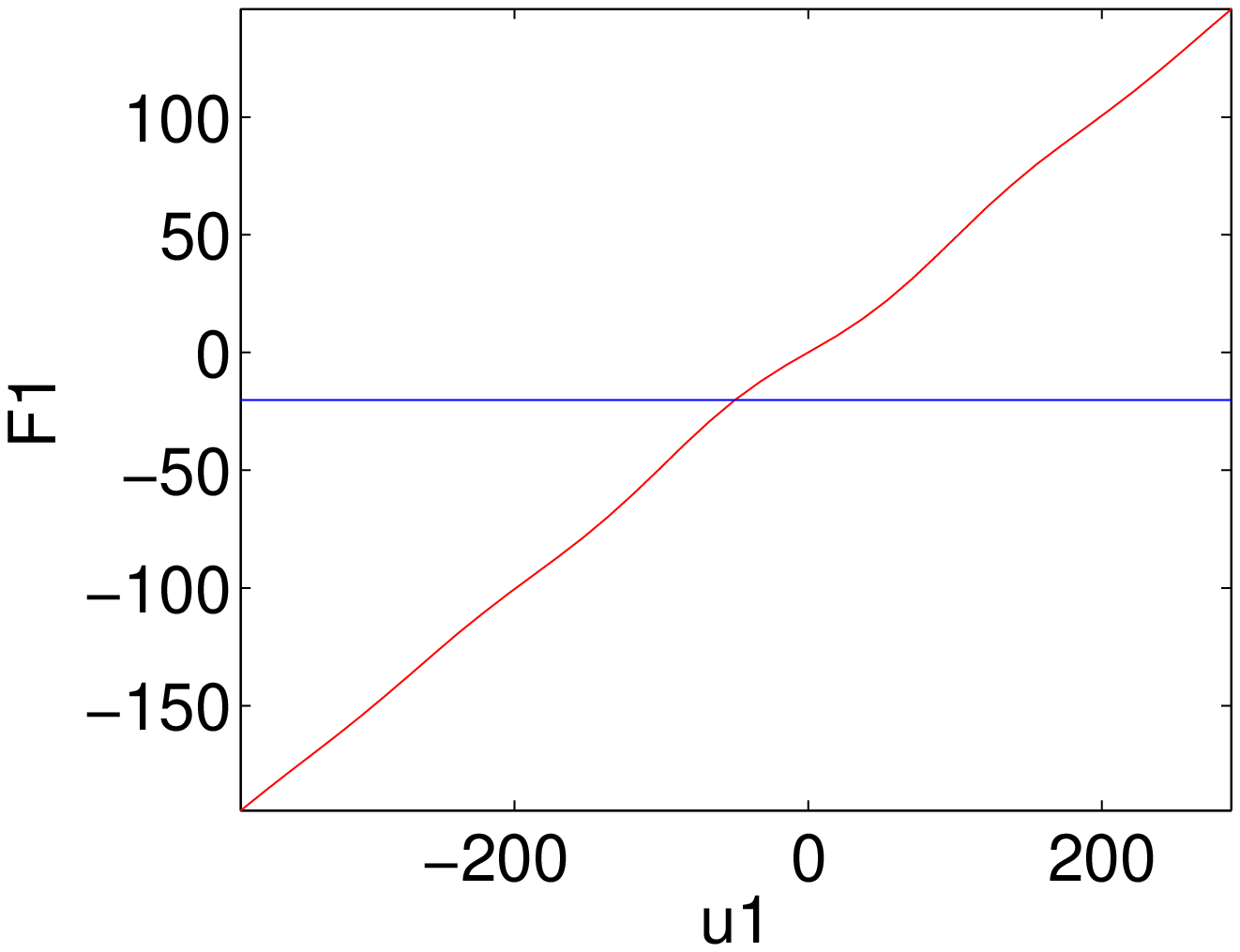} }}
\caption{$f'$ strictly below $\lambda_1$} \label{fig:hamsub1}
\end{figure}
\begin{figure}[h]
\centering
\mbox{\subfigure{\includegraphics[width=.45\linewidth]{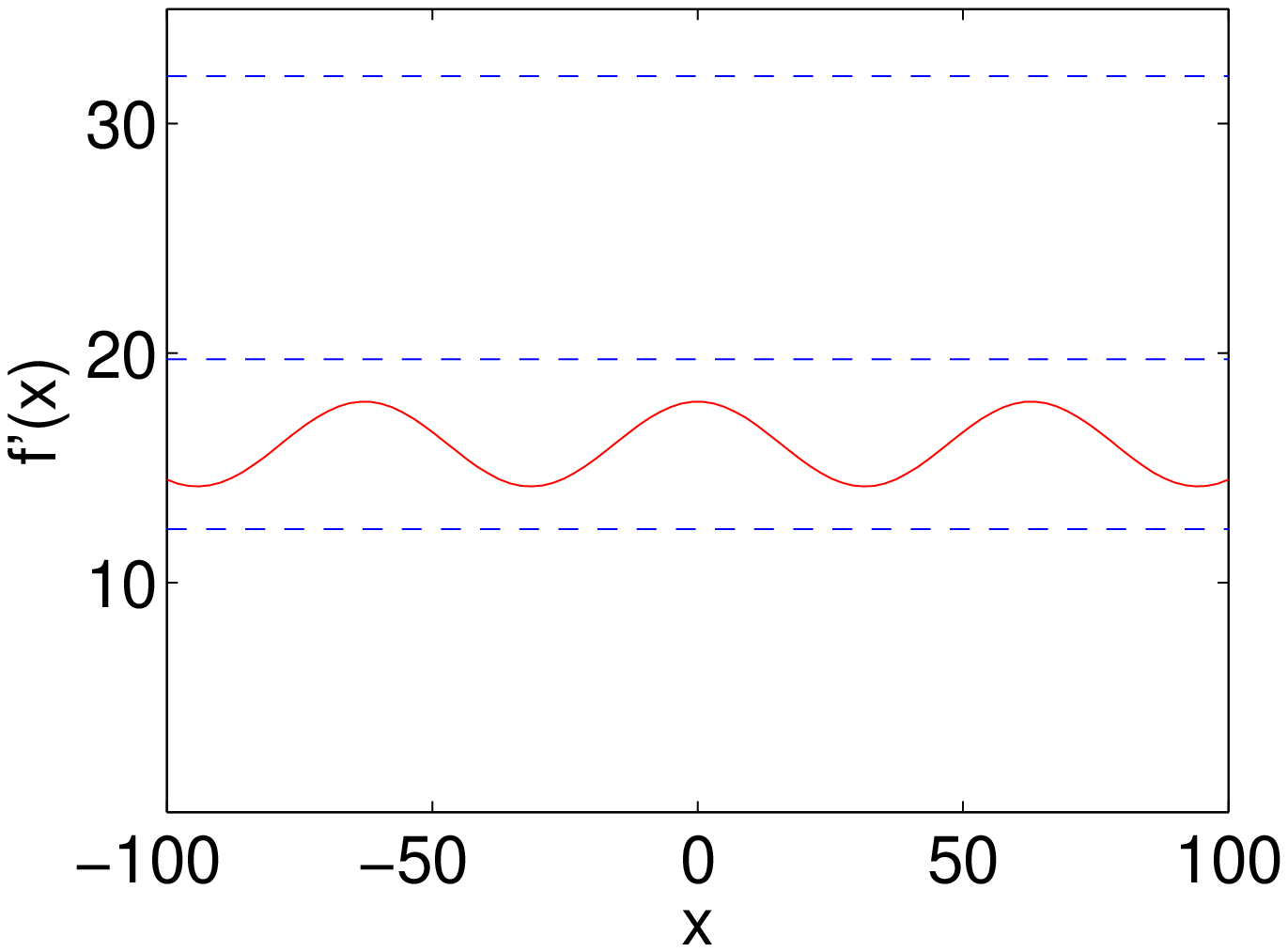}}\quad\quad
\subfigure{\includegraphics[width=.45\linewidth]{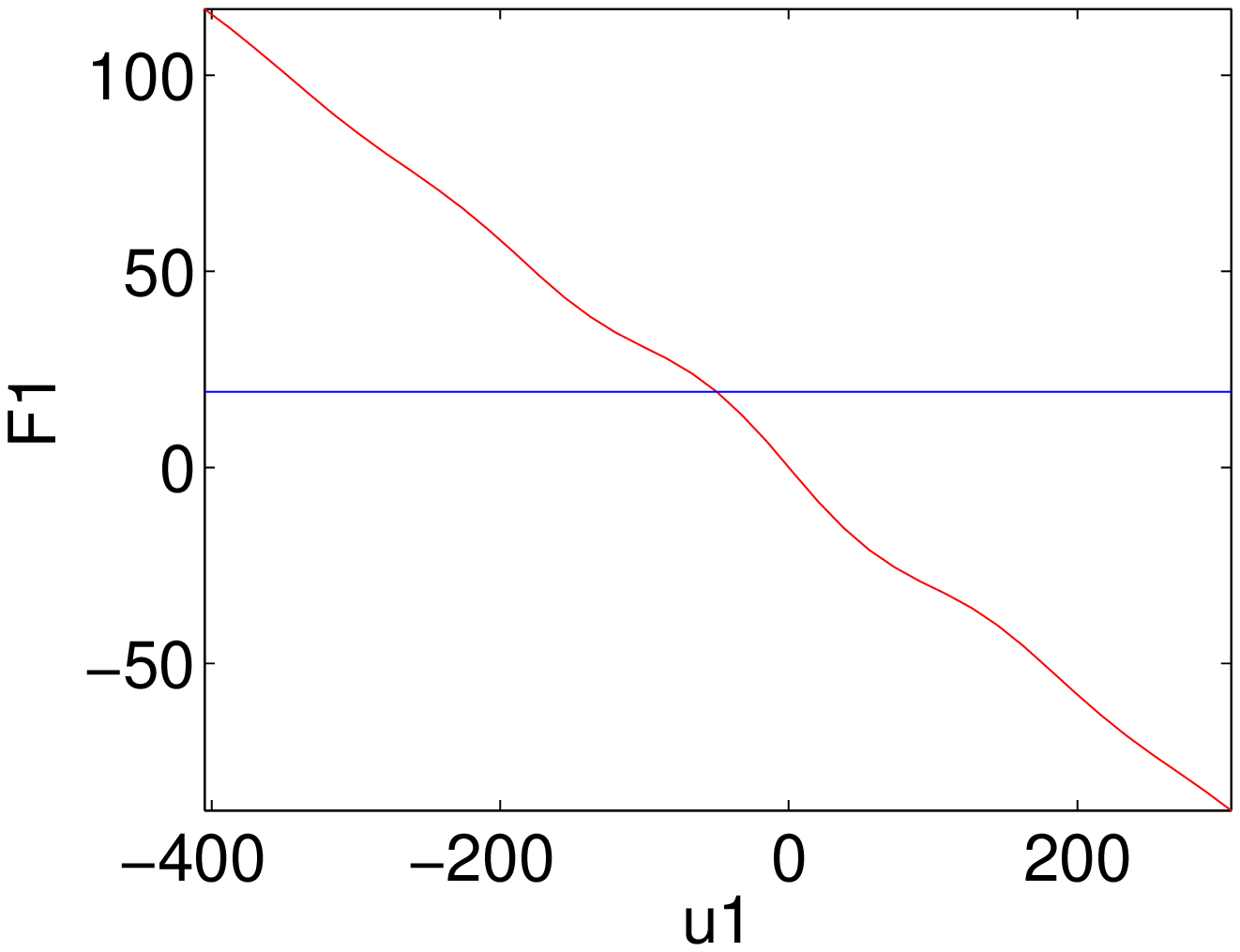} }}
\caption{$f'$ between $\lambda_1$ and $\lambda_2$}  \label{fig:hamsup1}
\end{figure}
The left of Figure \ref{fig:hamsub1} is the graph of $f'$: it lies below the first eigenvalue. The first three eigenvalues are marked as dotted lines.
The graph on the right illustrates the fact that as we move up along the fiber $u_1 = u_0 + t \varphi_1$, the corresponding point in the range $F_1 = F(u_1)$ also moves up.

Similarly, in Figure \ref{fig:hamsup1}, the derivative of $f$ lies strictly between $\lambda_1$ and $\lambda_2$. Here, moving up in the fiber, corresponds
to moving down in the range. The graphs are consistent with the fact that $F:X \to Y$ is a global diffeomorphism in both cases.

\subsubsection{Ambrosetti-Prodi}
We now return to the example of Section \ref{sec:movhor},
in which $\lambda_1$ is the only eigenvalue in $\overline{f'(\RR)}$. Naïvely, Figures \ref{fig:hamsub1} and \ref{fig:hamsup1} indicate that, as we move up in the fiber, the image under $F$ initially goes up, then down, as shown in Figure~\ref{fig:AP}.

\begin{figure}[h]
\centering
\mbox{\subfigure{\includegraphics[width=.45\linewidth]{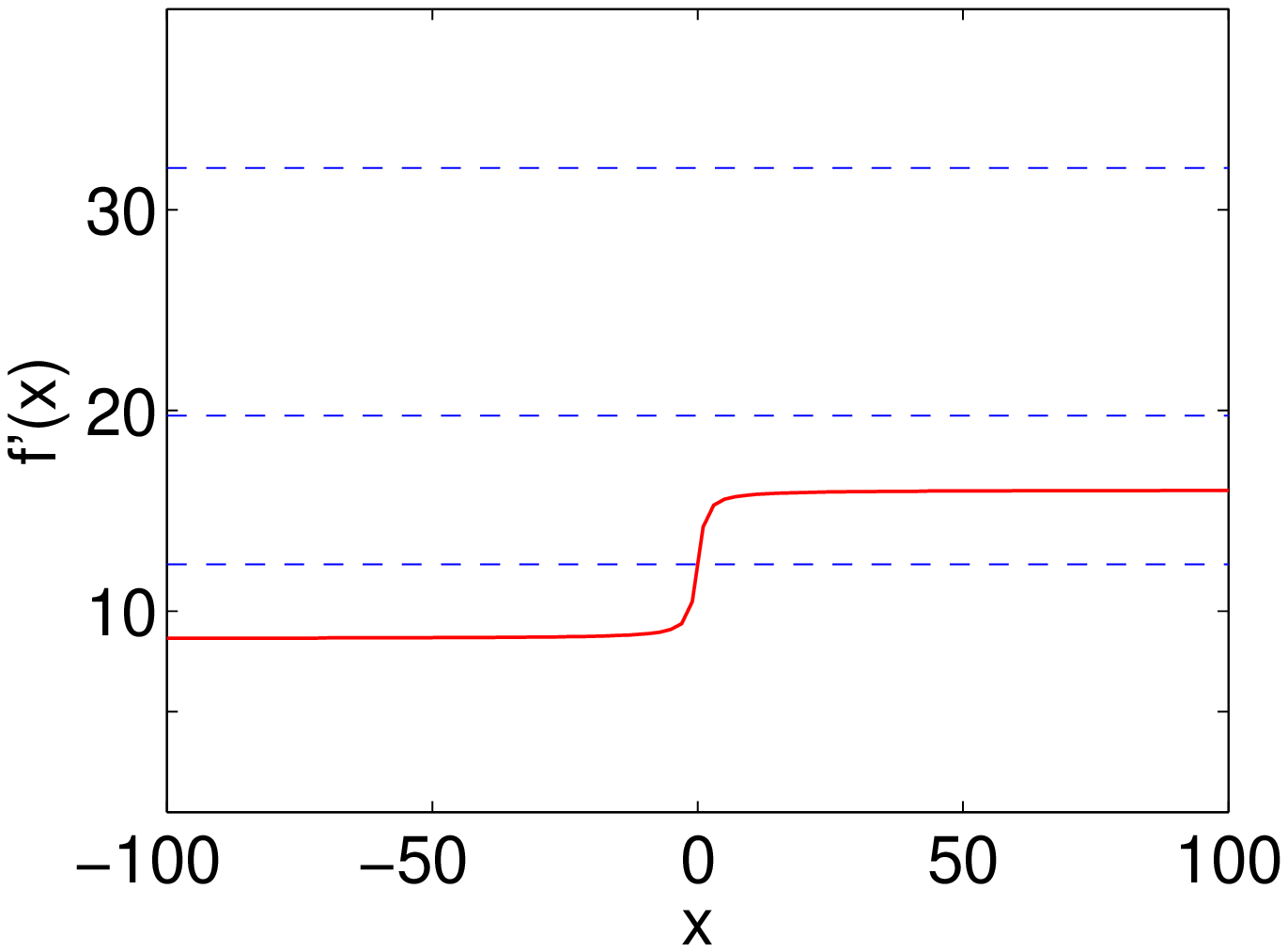}}\quad\quad
\subfigure{\includegraphics[width=.45\linewidth]{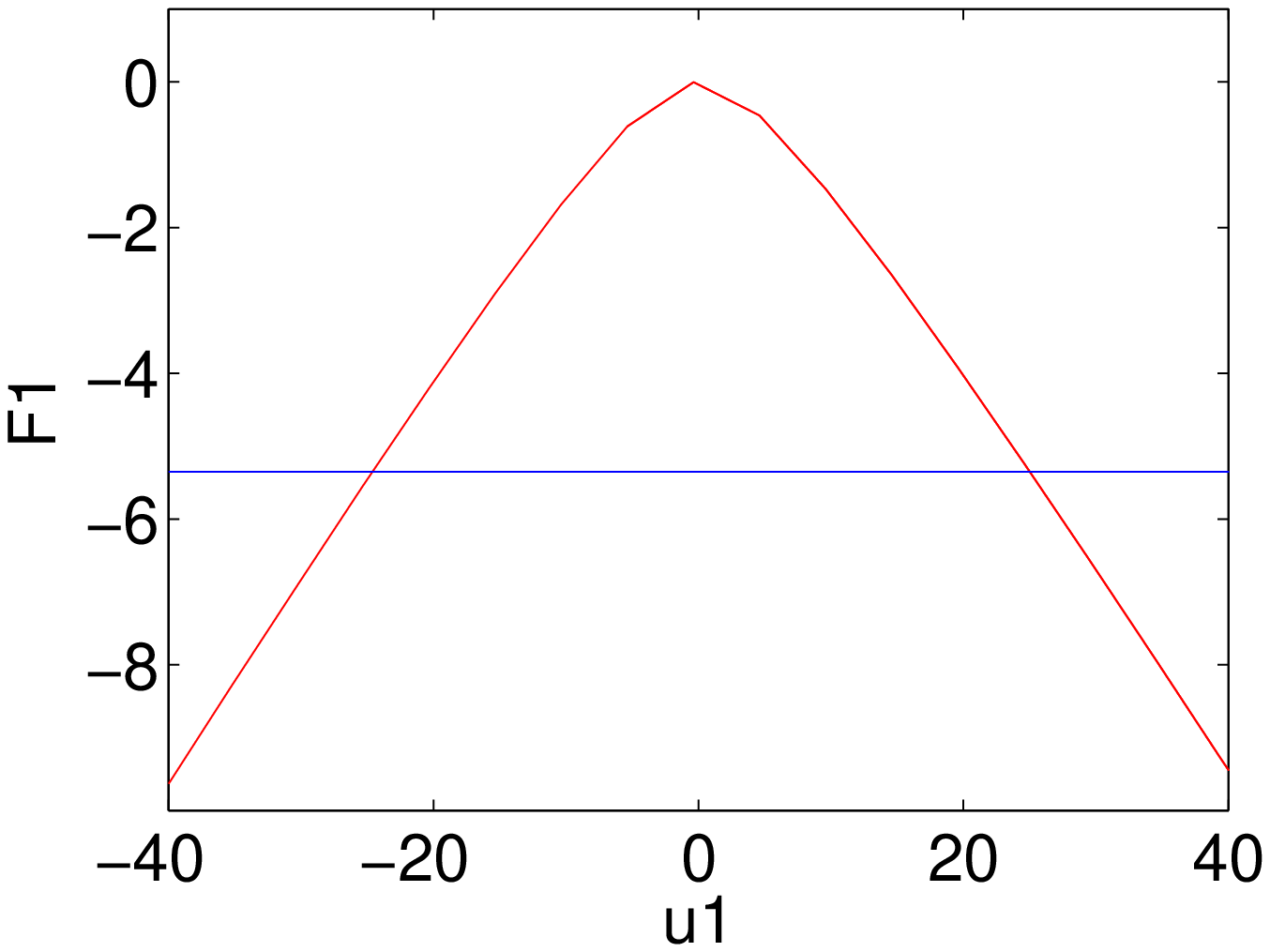} }}
\caption{$\overline{f'(\RR)}\cap\spectrum{-\lap}=\{\lambda_1\} = {\calK}$}  \label{fig:AP}
\end{figure}

Again, the picture is in agreement with the Ambrosetti-Prodi theorem: below a certain height, a point in the vertical line through $F(u_0)$ has two preimages.
The two preimages of $F(u_0)$ are shown in Figure \ref{fig:APsols}.
\begin{figure}[h]
\begin{center}
\subfigure{\includegraphics[width=.4\linewidth]{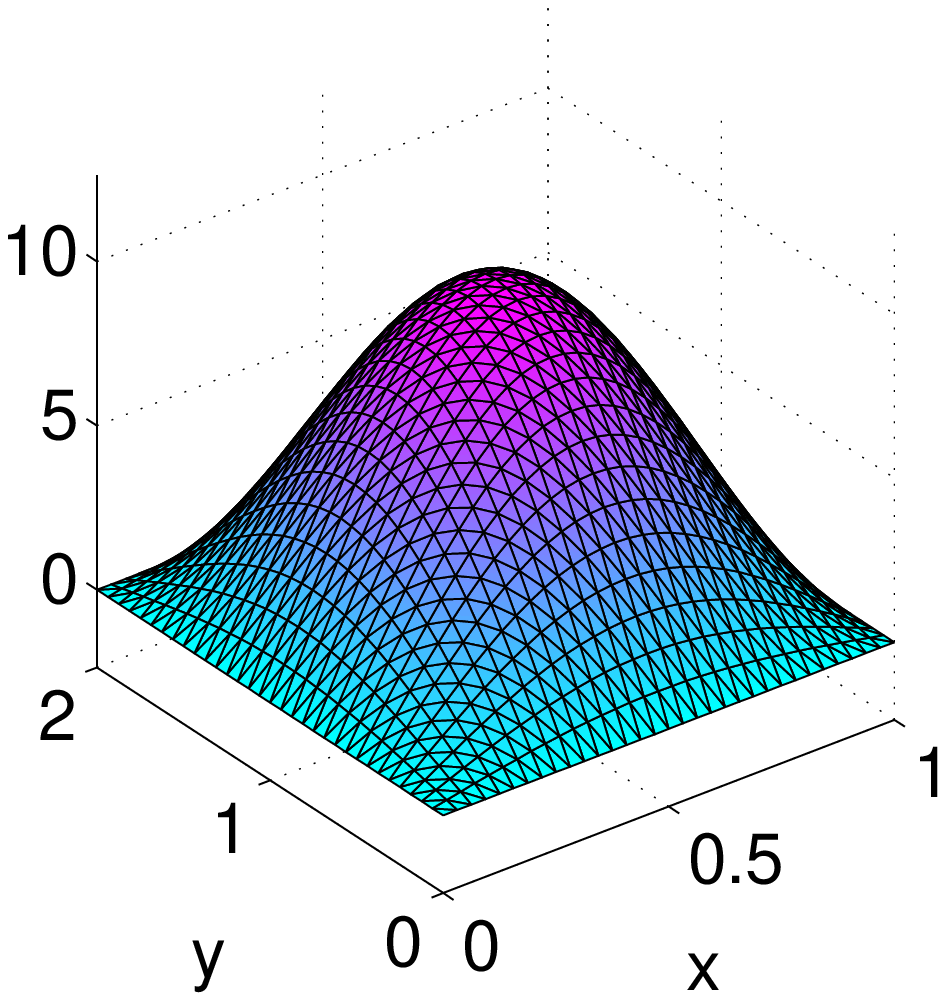}}\quad
\subfigure{\includegraphics[width=.4\linewidth]{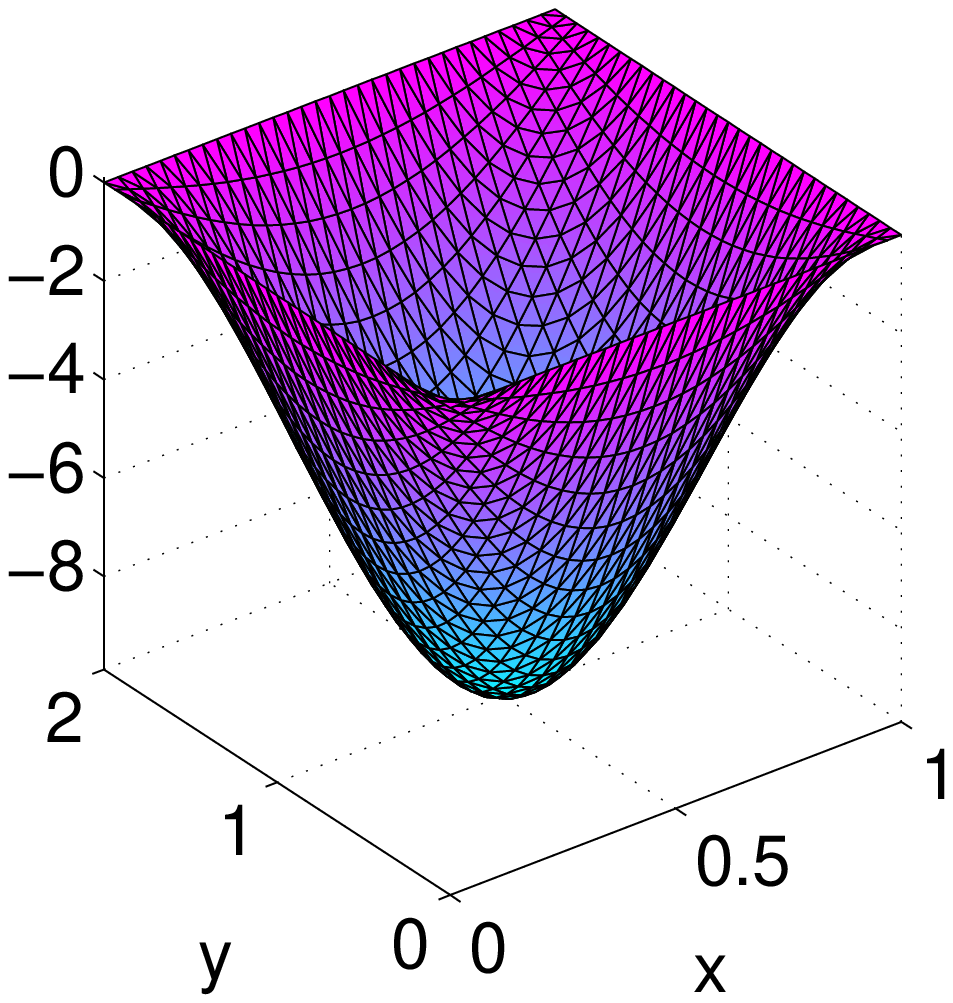} }
\end{center}
\caption{Ambrosetti-Prodi solutions}
\label{fig:APsols}
\end{figure}

\subsubsection{Non-convex $f$, ${\calK}= \{1\}$}
\begin{figure}[!h]
\centering
\mbox{\subfigure{\includegraphics[width=.45\linewidth]{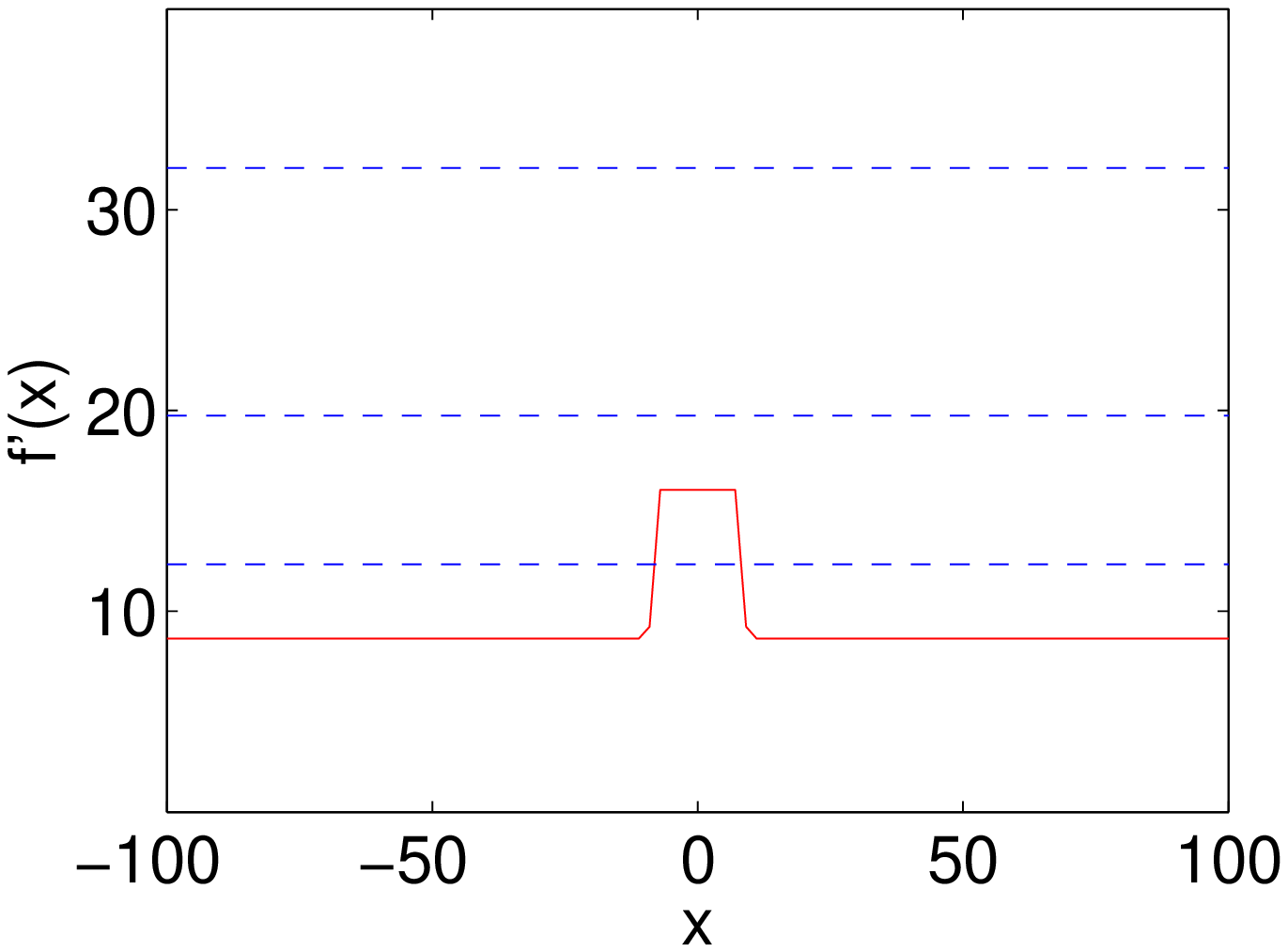}}\quad\quad
\subfigure{\includegraphics[width=.45\linewidth]{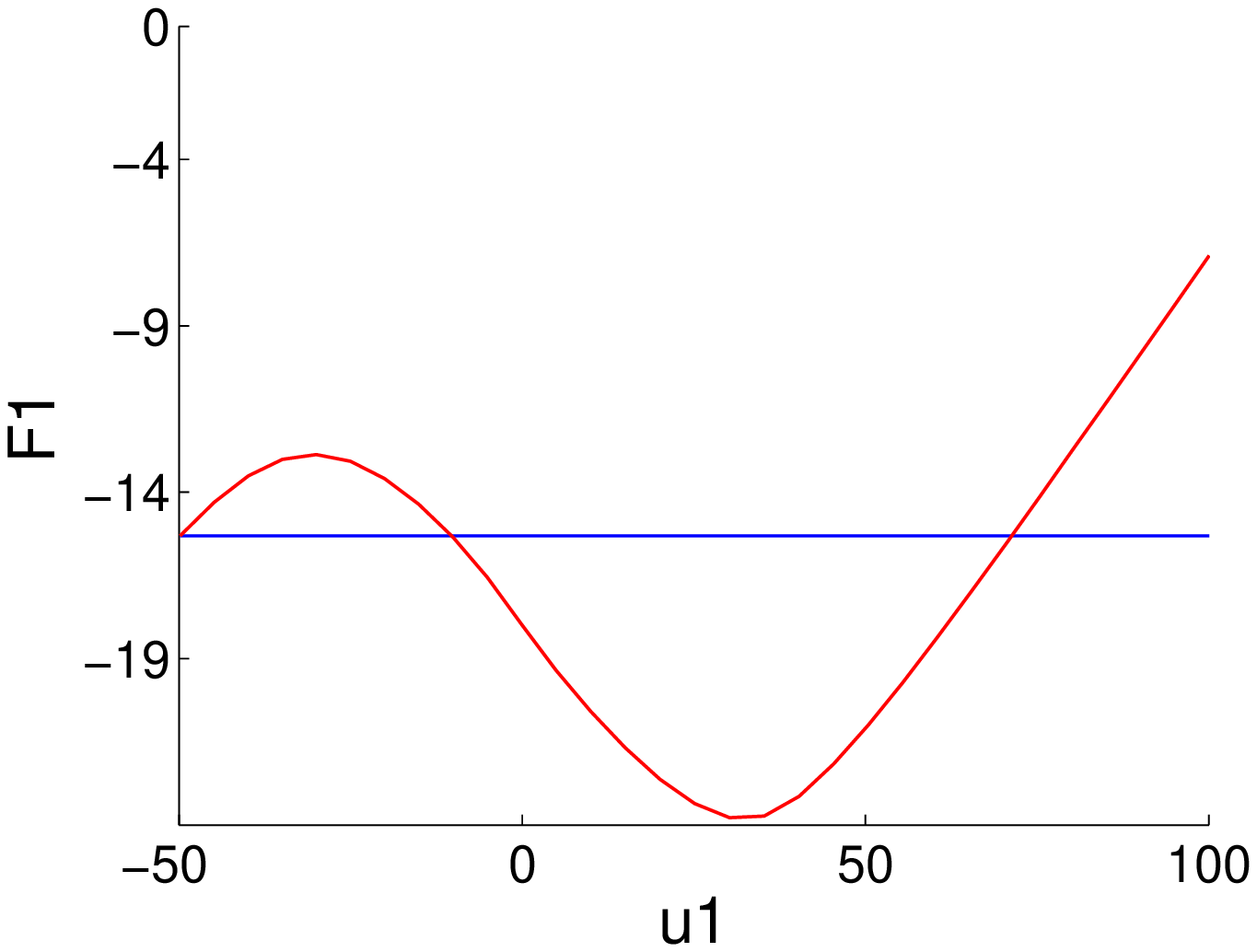} }}
\caption{Non-convex $f$}  \label{fig:NC}
\end{figure}
Things get more interesting if we relax the condition that $f$ be convex.
In Figure~\ref{fig:NC} we analyze the situation in which a non-convex $f'$ interacts only with $\lambda_1$.
For  $u_0=-50\,\varphi_1+10\,\varphi_2$ and $g=F(u_0)$,  the equation $F(u)=g$ has three distinct solutions, displayed in Figure~\ref{fig:NCsols}.
\begin{figure}[!h]
\begin{center}$
\begin{array}{ccc}
\subfigure{\includegraphics[width=.3\linewidth]{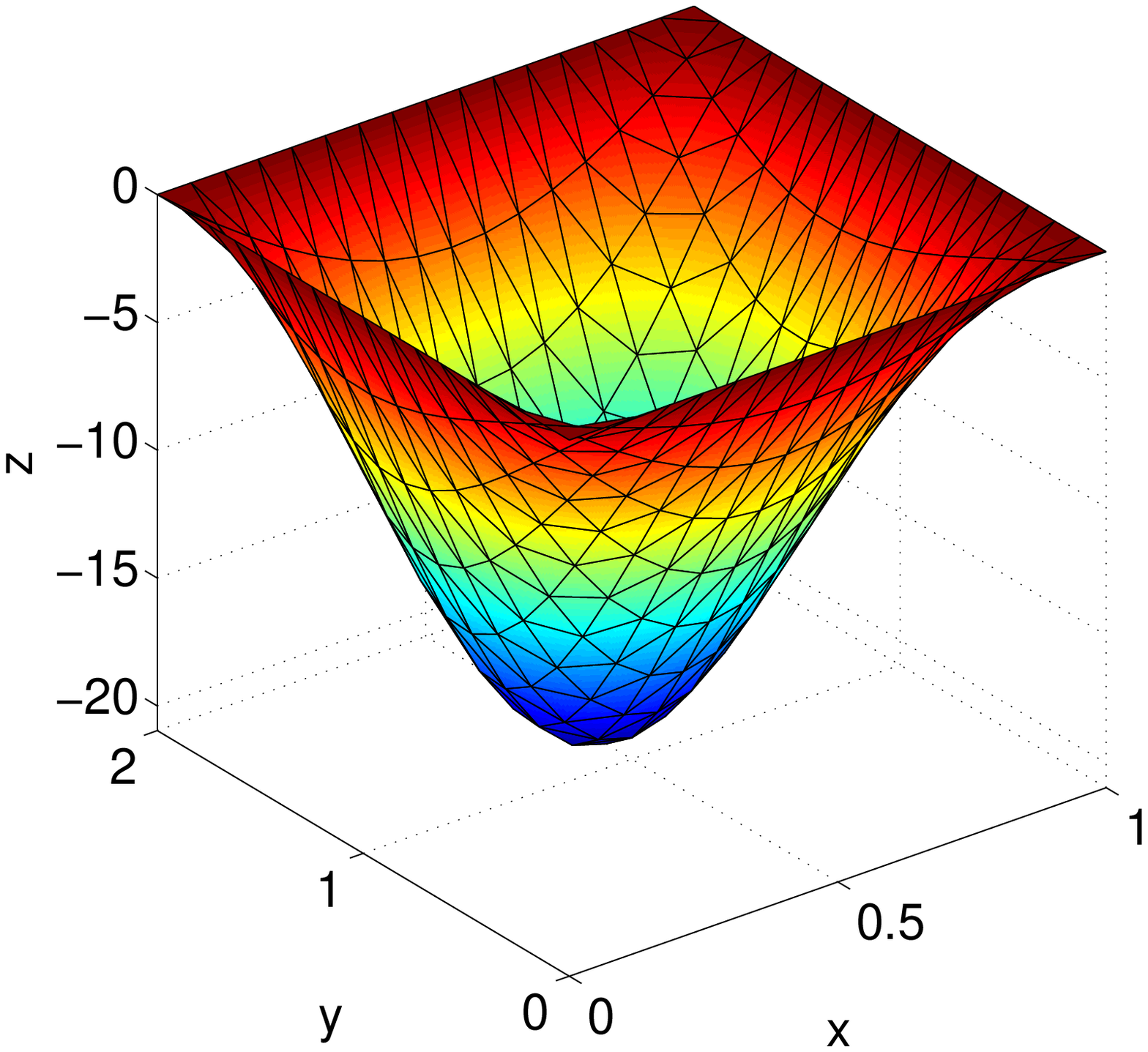}} &
\subfigure{\includegraphics[width=.3\linewidth]{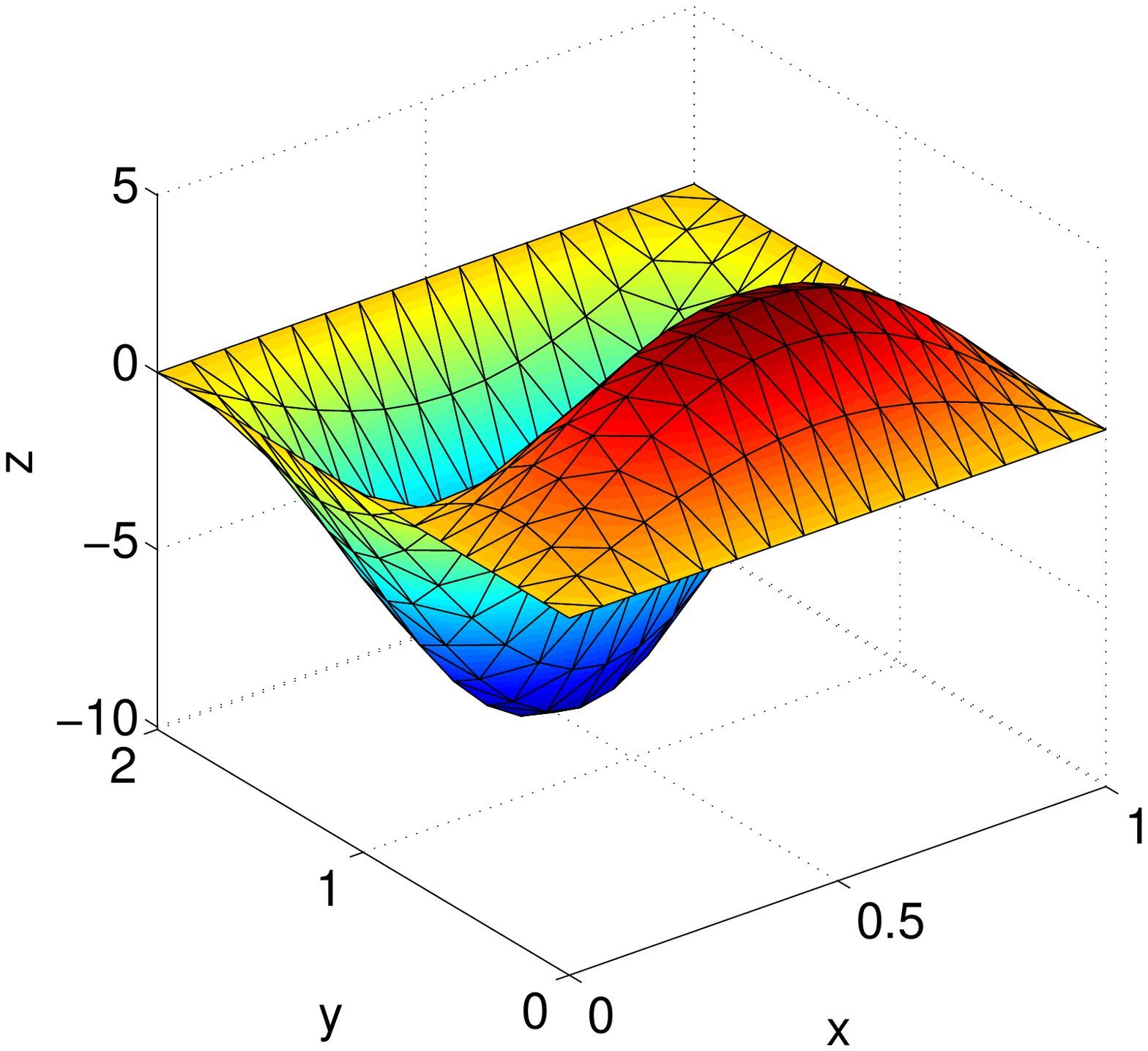}} &
\subfigure{\includegraphics[width=.3\linewidth]{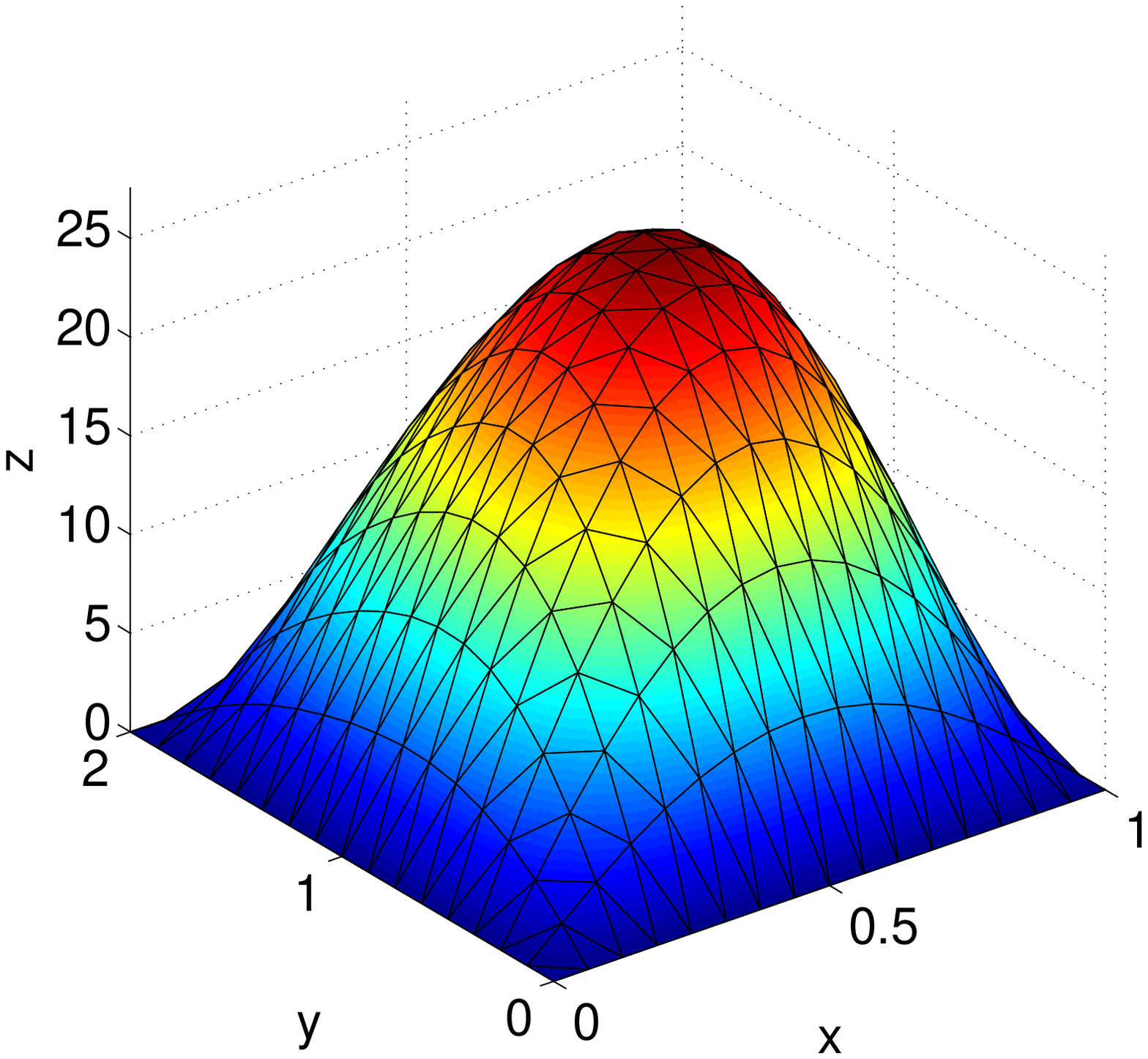}}
\end{array}$
\end{center}
\caption{The three solutions}
\label{fig:NCsols}
\end{figure}

The frames in Figure \ref{fig:NCfibseq} show that the action of $F$ on fibers is \emph{not} homogeneous. The plots show the images under $F$ of fibers $\alpha_{g_i}$ with $g_i=F(-50\varphi_1+c_i\varphi_2)$, for $c_1=10$ (same as Fig. \ref{fig:NC}), $c_2=45$ and $c_3=100$.
\begin{figure}[!h]
\begin{center}
\includegraphics[width=.3\linewidth]{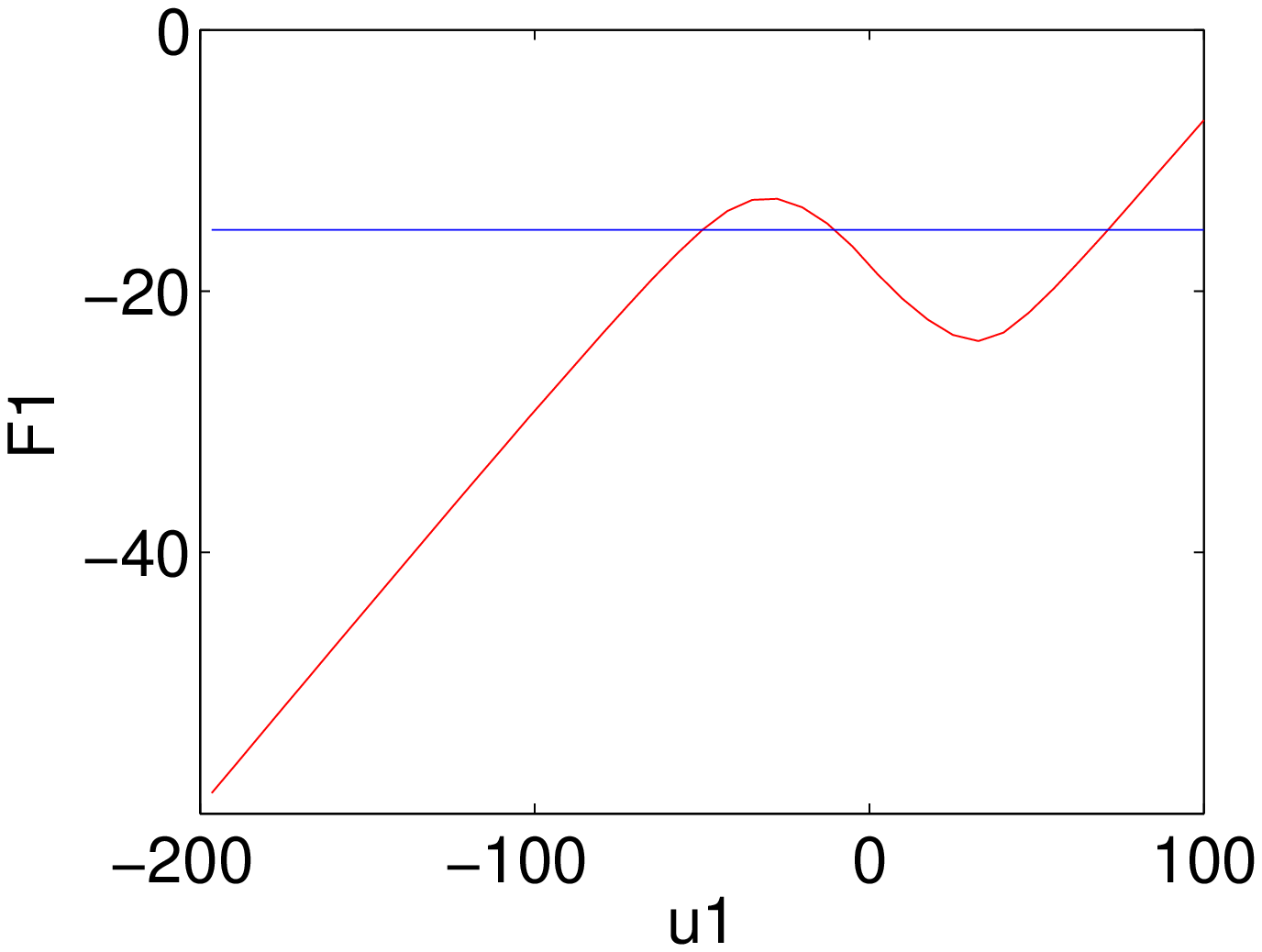}
\includegraphics[width=.3\linewidth]{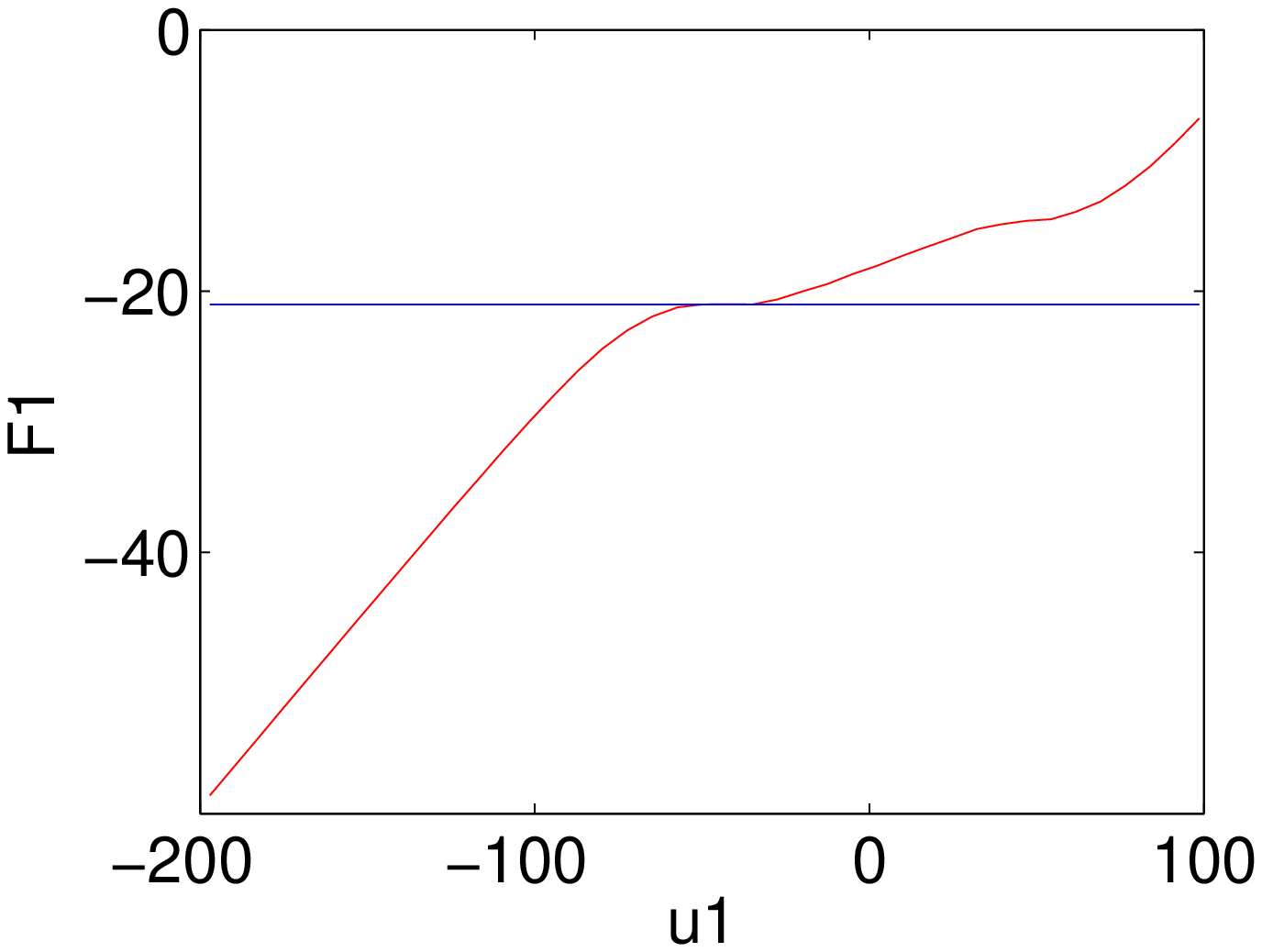}
\includegraphics[width=.3\linewidth]{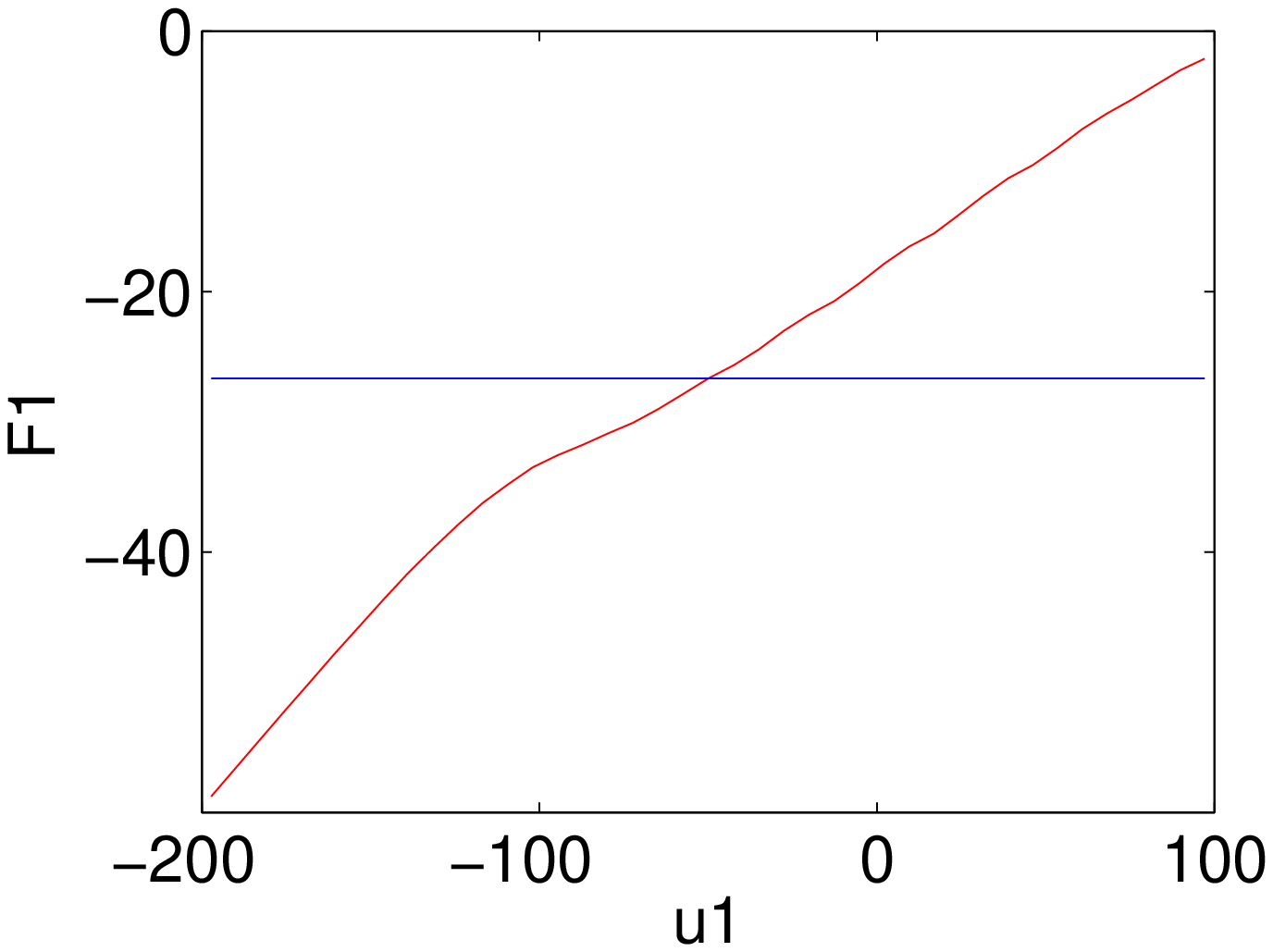}
\caption{Fibers getting mapped non-uniformly}
\label{fig:NCfibseq}
\end{center}
\end{figure}

\subsubsection{Convex $f$, $\mathbf{  {\calK}=\{2\}}$}
\begin{figure}[!h]
\centering
\mbox{\subfigure{\includegraphics[width=.45\linewidth]{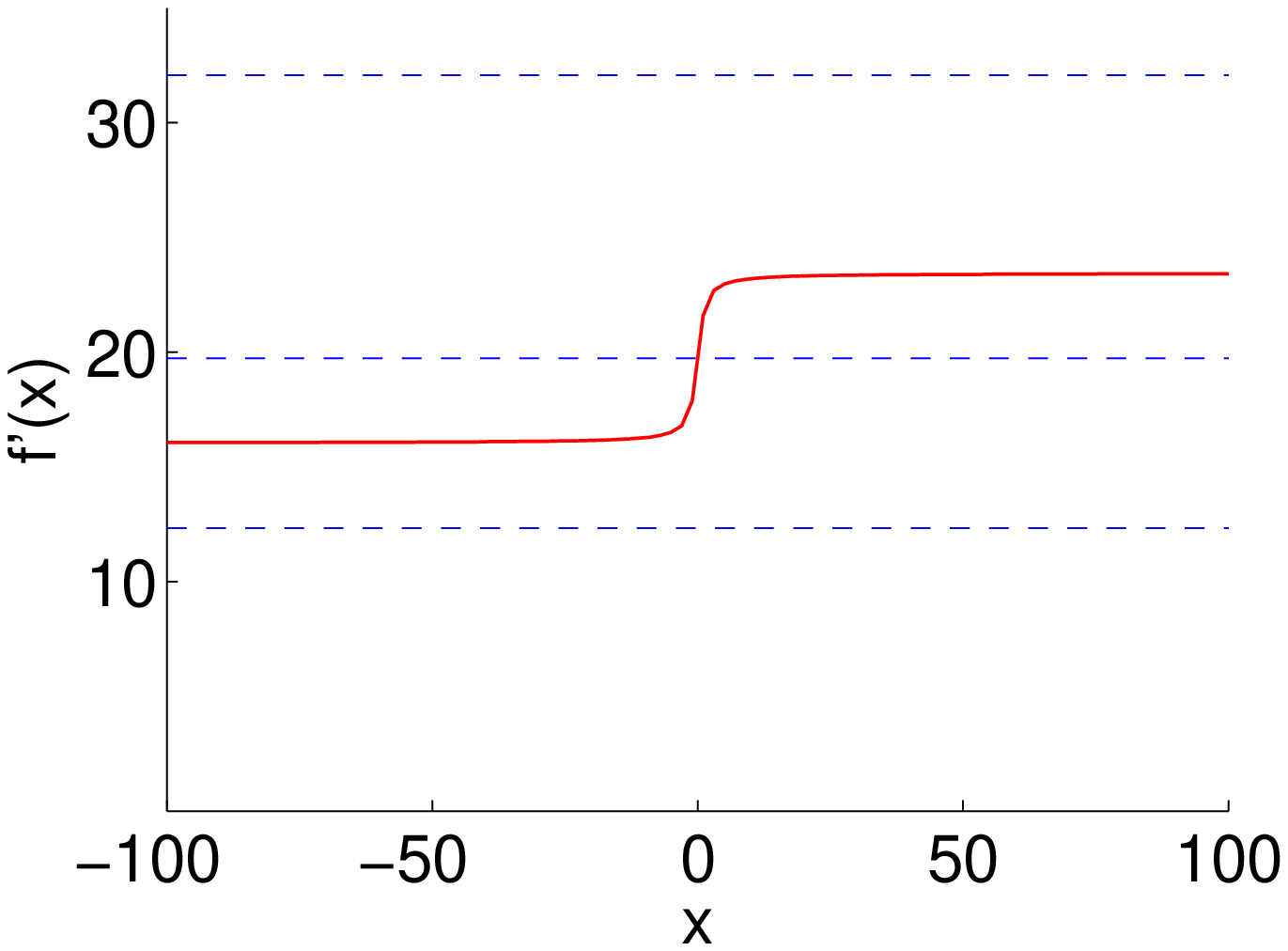}}\quad\quad
\subfigure{\includegraphics[width=.45\linewidth]{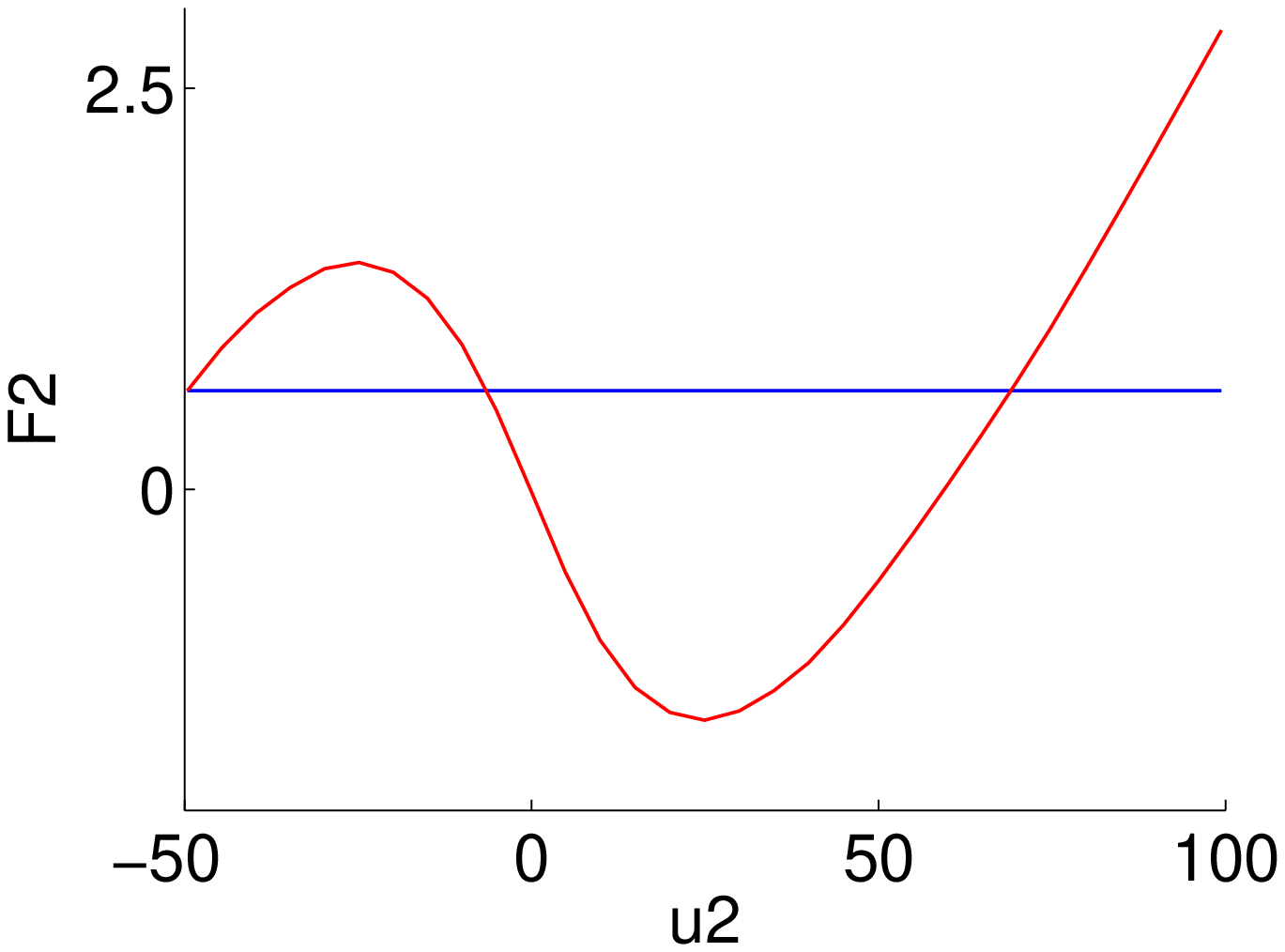} }}
\caption{$\overline{f'(\RR)}\cap\spectrum{-\lap}=\{\lambda_2\} = {\calK}$}  \label{fig:APl2}
\end{figure}
We take $f$ convex,
$\Range f' = \left( \lambda_2 - \frac{\lambda_2 - \lambda_1}{2}\, ,\,
\lambda_2 + \frac{\lambda_2 - \lambda_1}{2}\right)$
and $u_0=-50\,\varphi_2+10\,\varphi_1$.
In Figure \ref{fig:APl2}, heights along the fiber and its image are measured with respect to the second eigenfunction $\varphi_2$.
Now, for $g=F(u_0)$, there are \emph{three} preimages, shown in
Figure~\ref{fig:APl2sols}.
Numerical evidence suggests uniform action of $F$  across fibers.

\begin{figure}[!h]
\begin{center}$
\begin{array}{ccc}
\subfigure{\includegraphics[width=.3\linewidth]{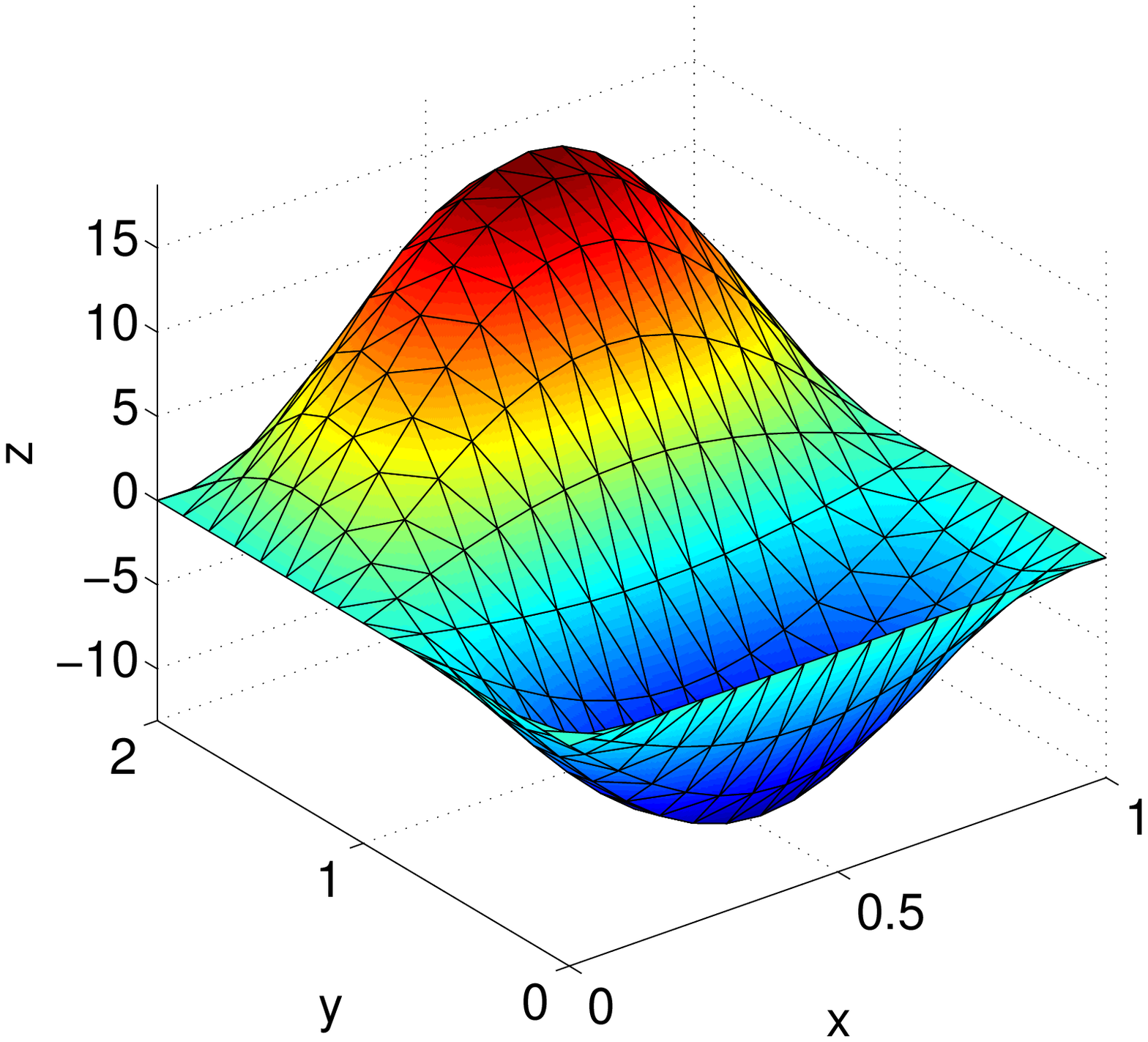}} &
\subfigure{\includegraphics[width=.3\linewidth]{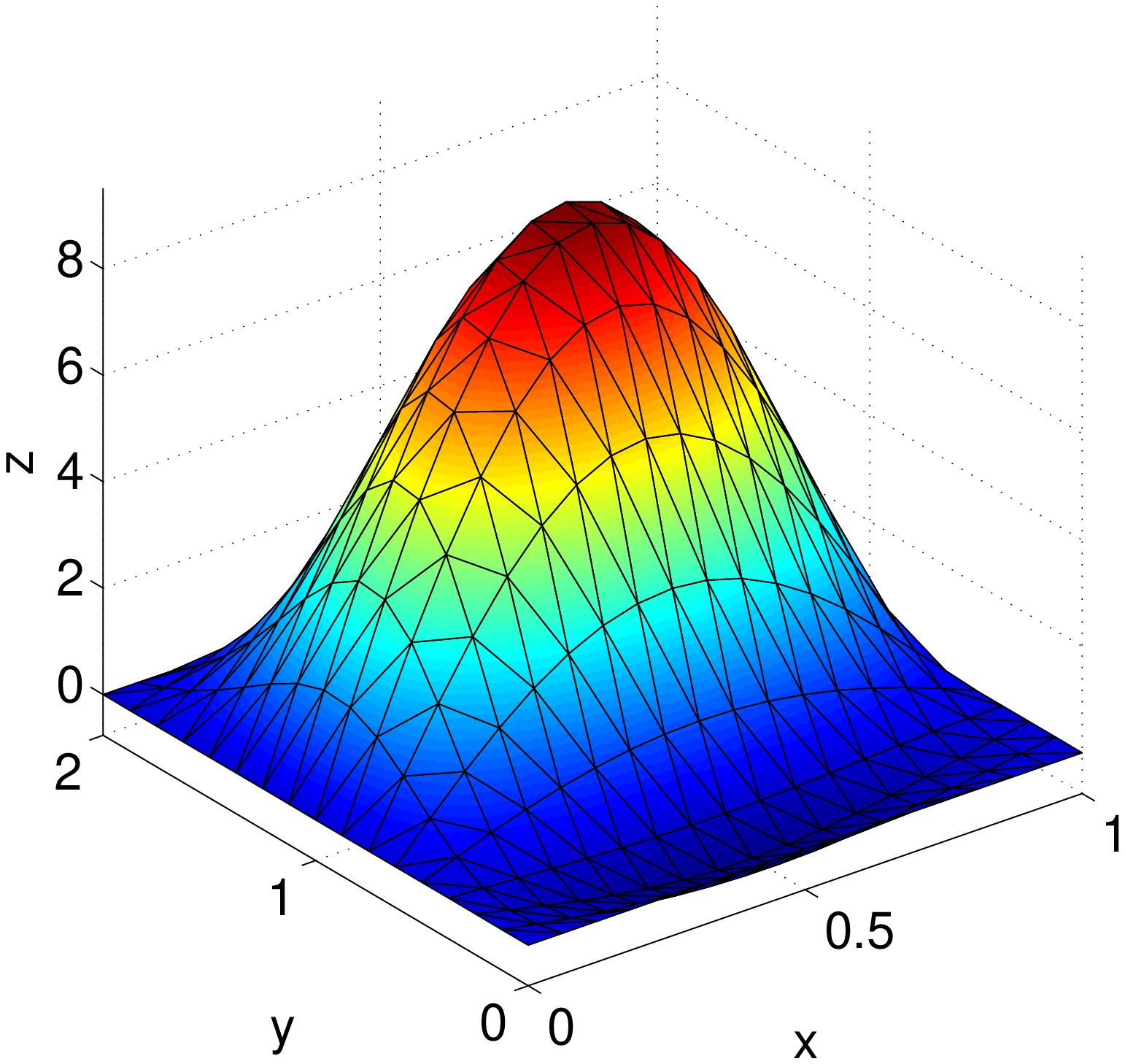}} &
\subfigure{\includegraphics[width=.3\linewidth]{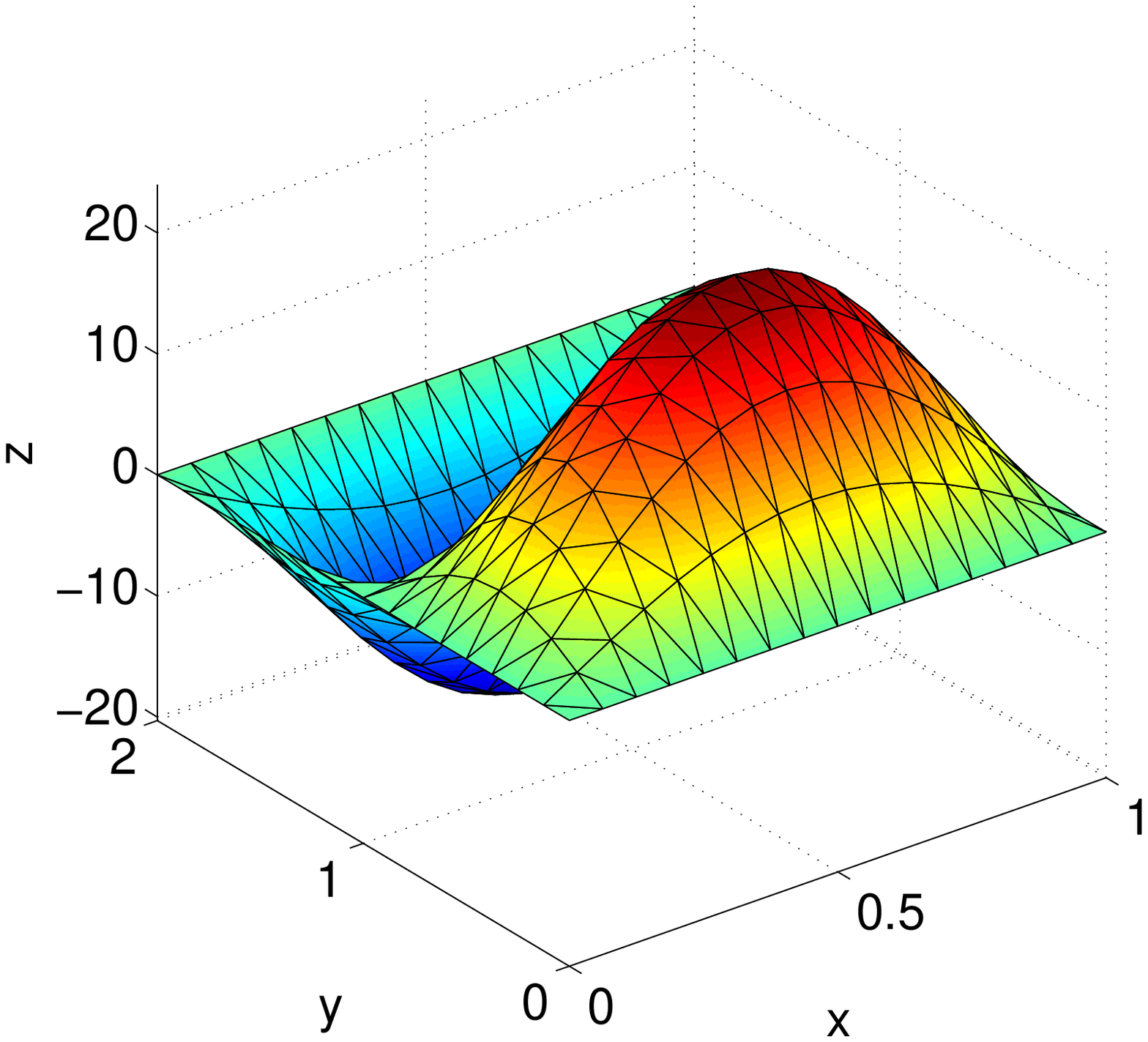}}
\end{array}$
\end{center}
\caption{Convex $f$, $\calK=\{2\}$: the three solutions}
\label{fig:APl2sols}
\end{figure}

\subsection{A two dimensional fiber: ${\calK}= \{1,2\}$}\label{subsec:int2lams}
We now try to visualize the action of $F$ on a two-dimensional fiber.
\begin{figure}[!h]
\begin{center}
\includegraphics[width=.5\linewidth]{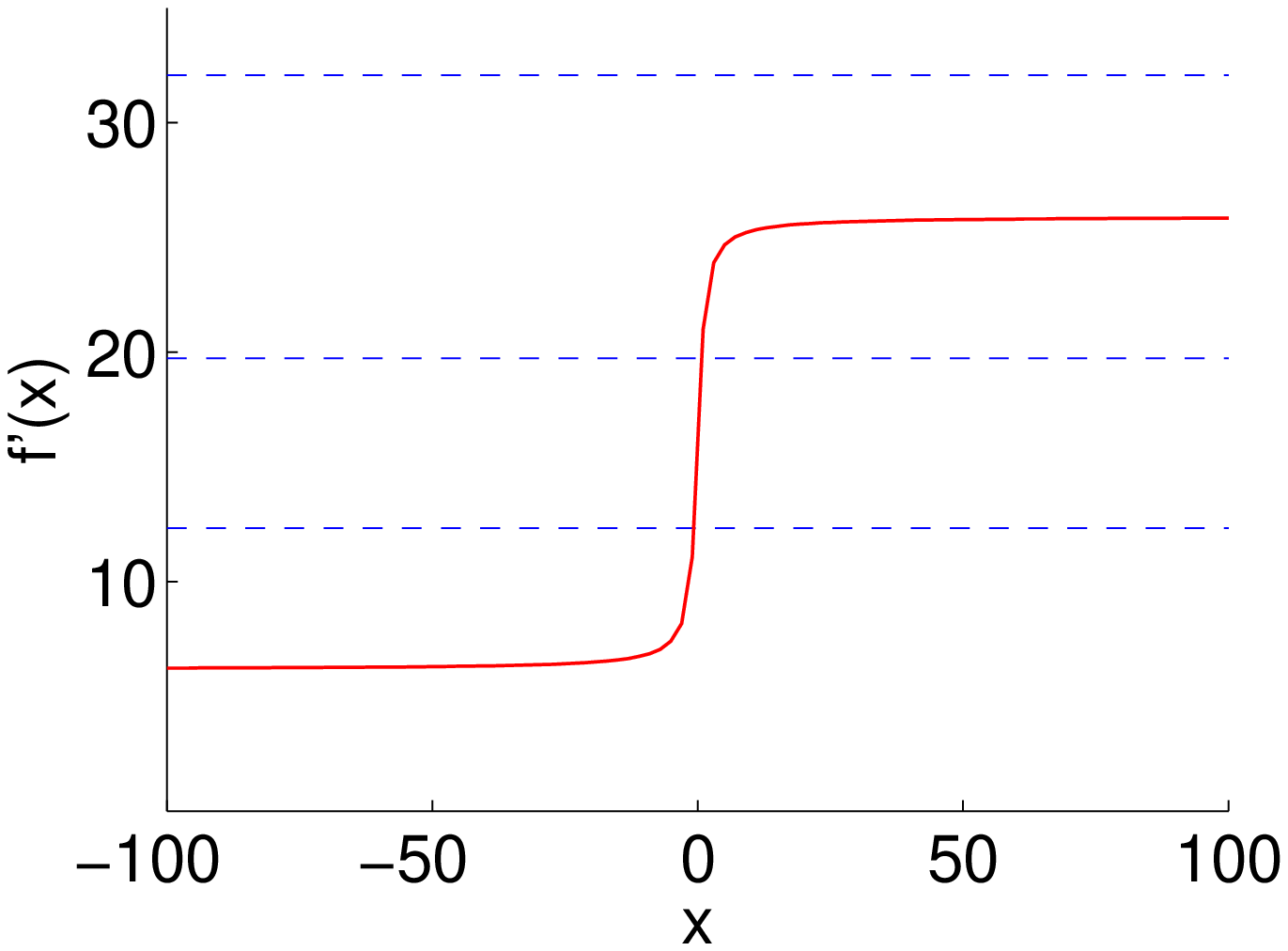}
\end{center}
\caption{Convex $f$, $\calK=\{1, 2\}$}
\label{fig:df2eigs}
\end{figure}

More specifically, we examine the fiber $\alpha_0$ through the zero function,
for which $F(0) = 0$.
Consider the circle $C$ in $\Vp$, the vertical plane spanned by $\varphi_1$ and $\varphi_2$, shown in Figure \ref{fig:UpDown}.
Let $C_{\alpha} \subset\alpha_0$ be the curve $\mathcal{H}_0(C)$, which projects bijectively under
$\Qp$ to $C$.
\begin{figure}[!h]
\begin{center}$
\begin{array}{cc}
\subfigure{\includegraphics[width=.5\linewidth]{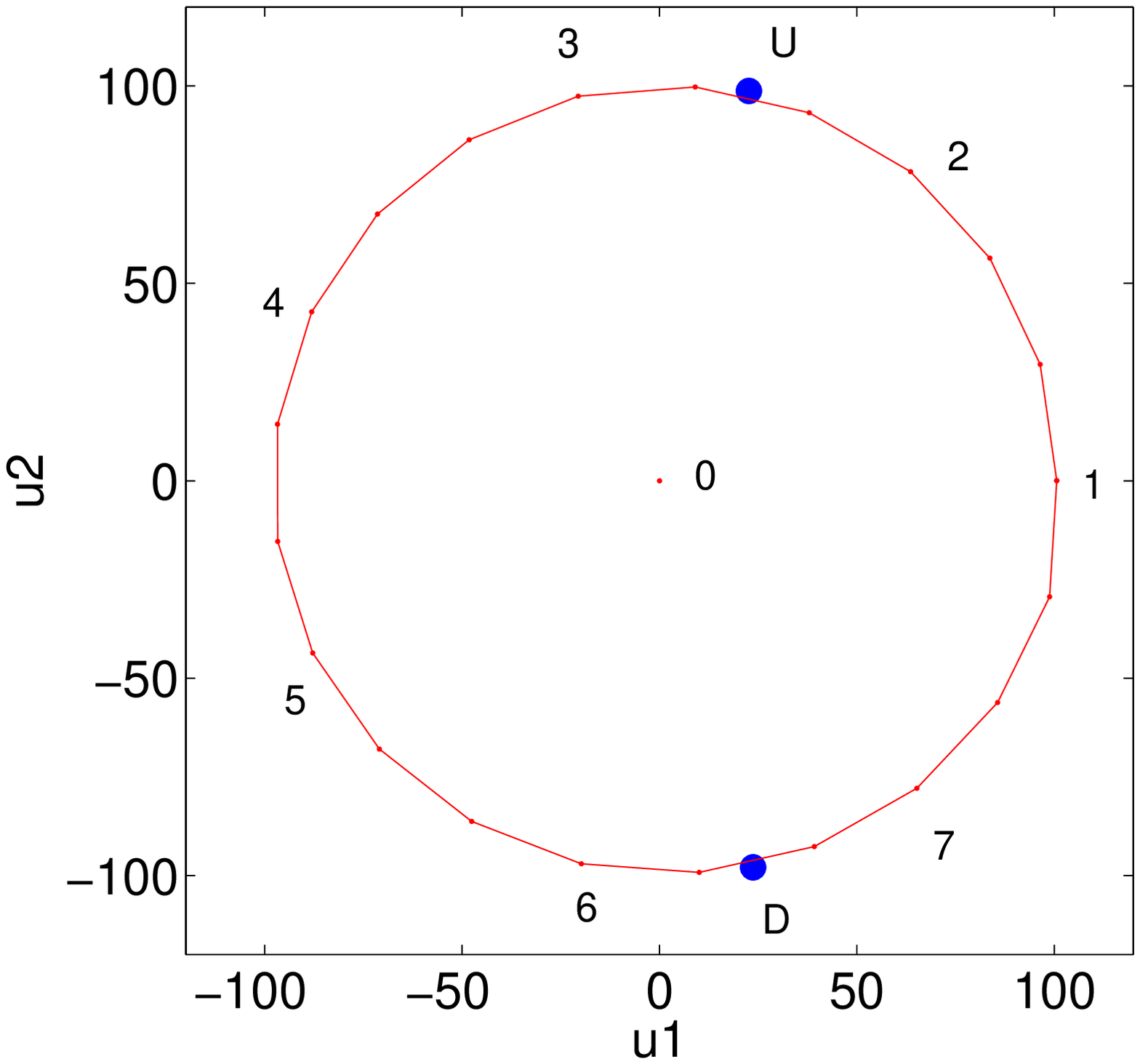}} &
\subfigure{\includegraphics[width=.5\linewidth]{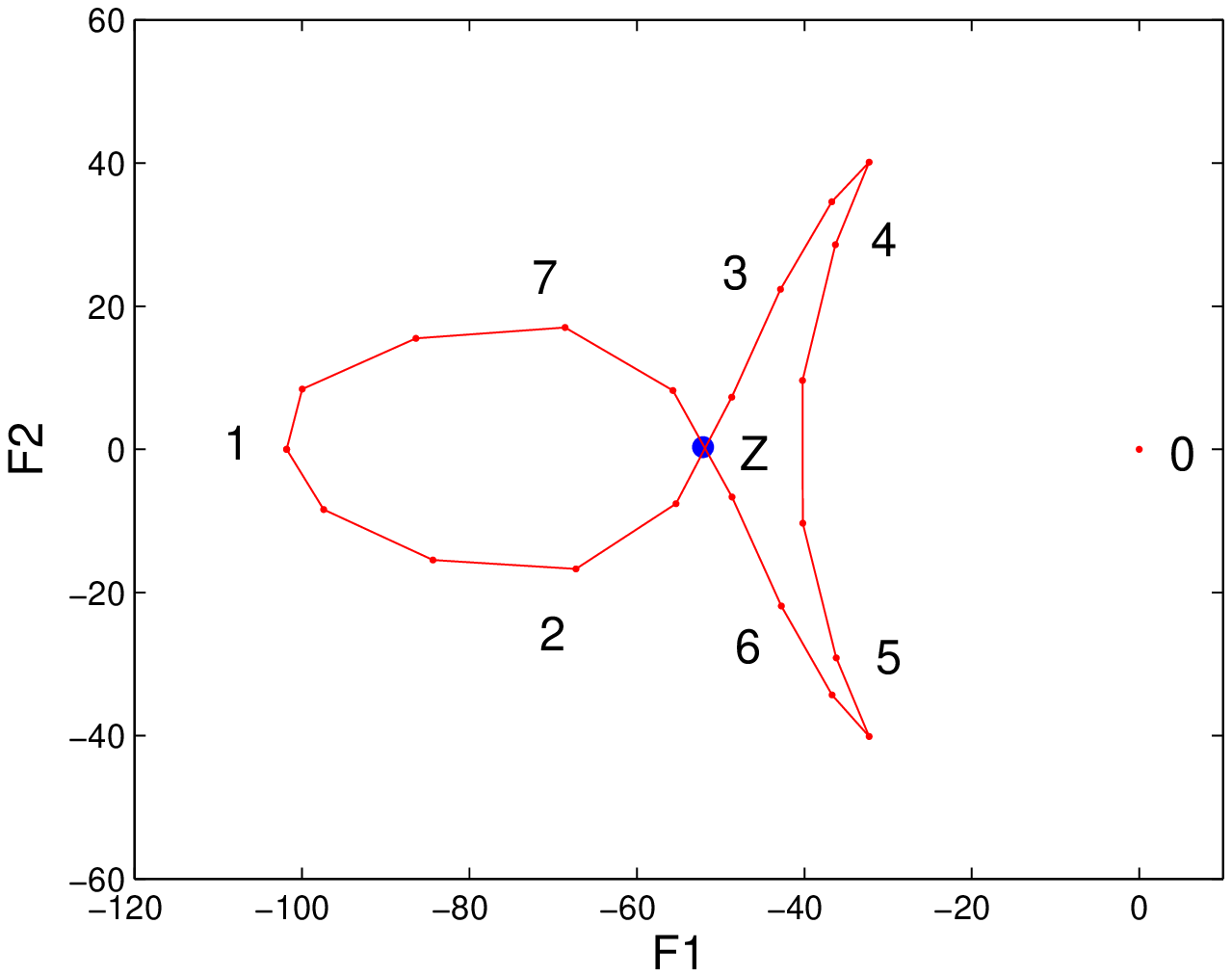}}
\end{array}$
\end{center}
\caption{Solutions U and D on the circle $C$ and images.}
\label{fig:UpDown}
\end{figure}

The fish-shaped curve in Figure \ref{fig:UpDown} is the projection of $F(C_{\alpha})$ under $\Qd$ in \Vd. Seven points and their images were given common labels. Let $g$, marked with a bullet, be the point of self-intersection of this curve.
Clearly $g$ has two preimages $U$ and $D$ between points 2 and 3
and 6 and 7, respectively.
Radial lines in the domain from the origin to points in $C$ give rise to lines from $F(0)=0$ to points in $F(C_{\alpha})$, as seen in Figure \ref{fig:LeftRight}. We
then obtain two approximate preimages $L$ and $R$ along the horizontal axis.

The four approximate preimages were then taken as initial guesses for Newton's Method and the four computed solutions are illustrated in Figure \ref{fig:UDLR}.
\begin{figure}[!h]
\begin{center}$
\begin{array}{cc}
\subfigure{\includegraphics[width=.5\linewidth]{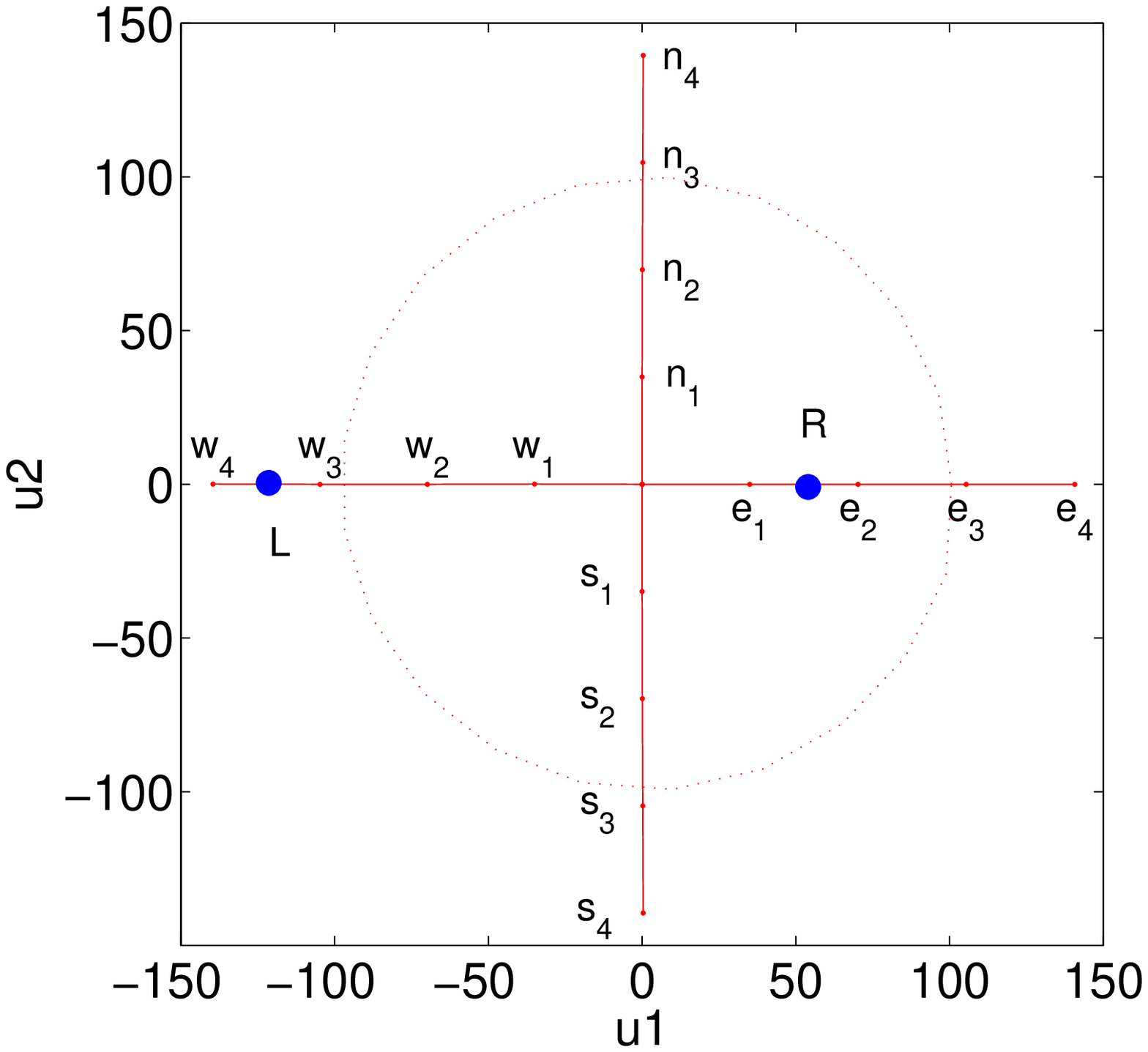}} &
\subfigure{\includegraphics[width=.5\linewidth]{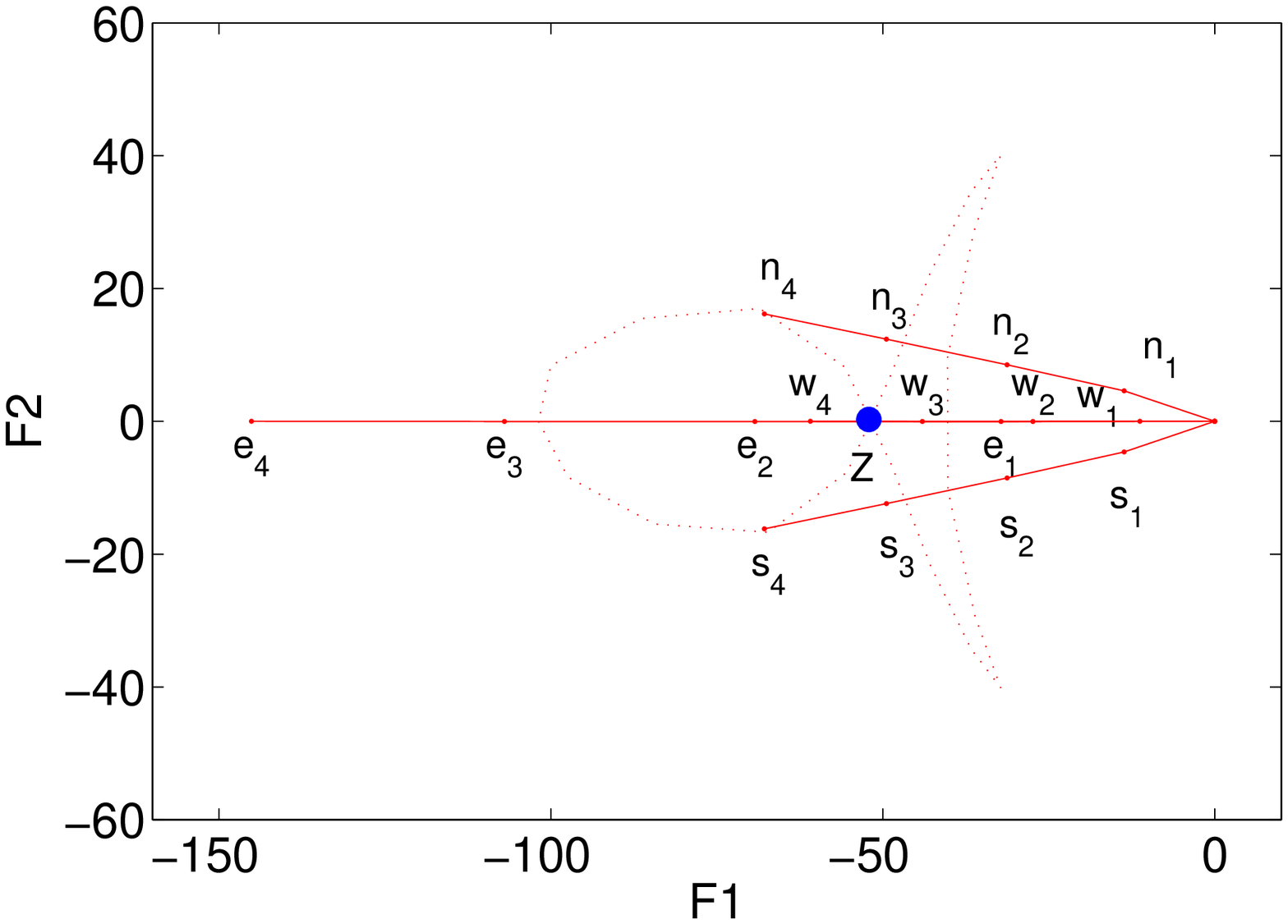}}
\end{array}$
\end{center}
\caption{Solutions L and  R on the u1 axis}
\label{fig:LeftRight}
\end{figure}

\begin{figure}[!h]
\begin{center}
\subfigure[$U$ and $D$]{\includegraphics[width=.4\linewidth]{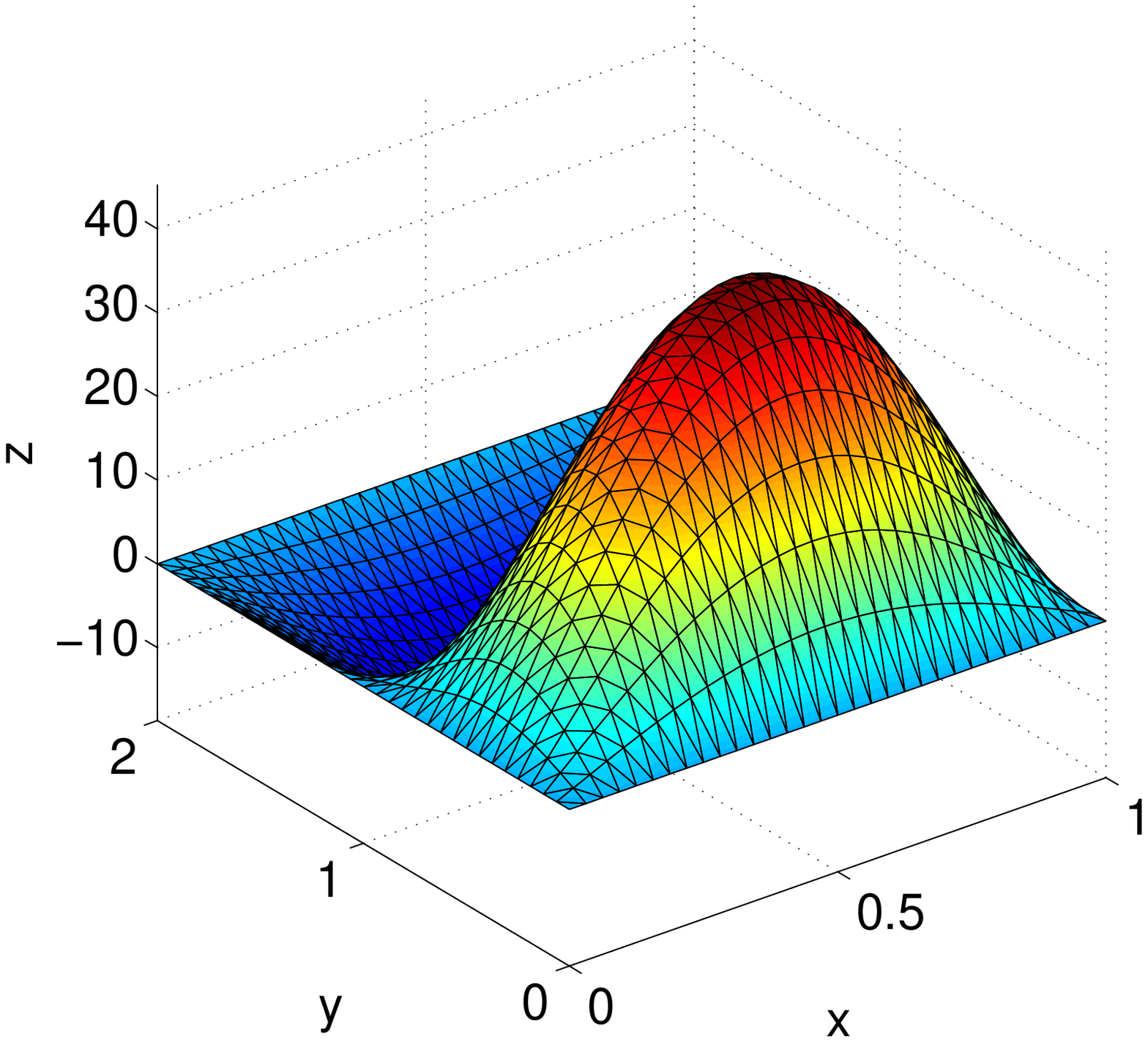}
\includegraphics[width=.4\linewidth]{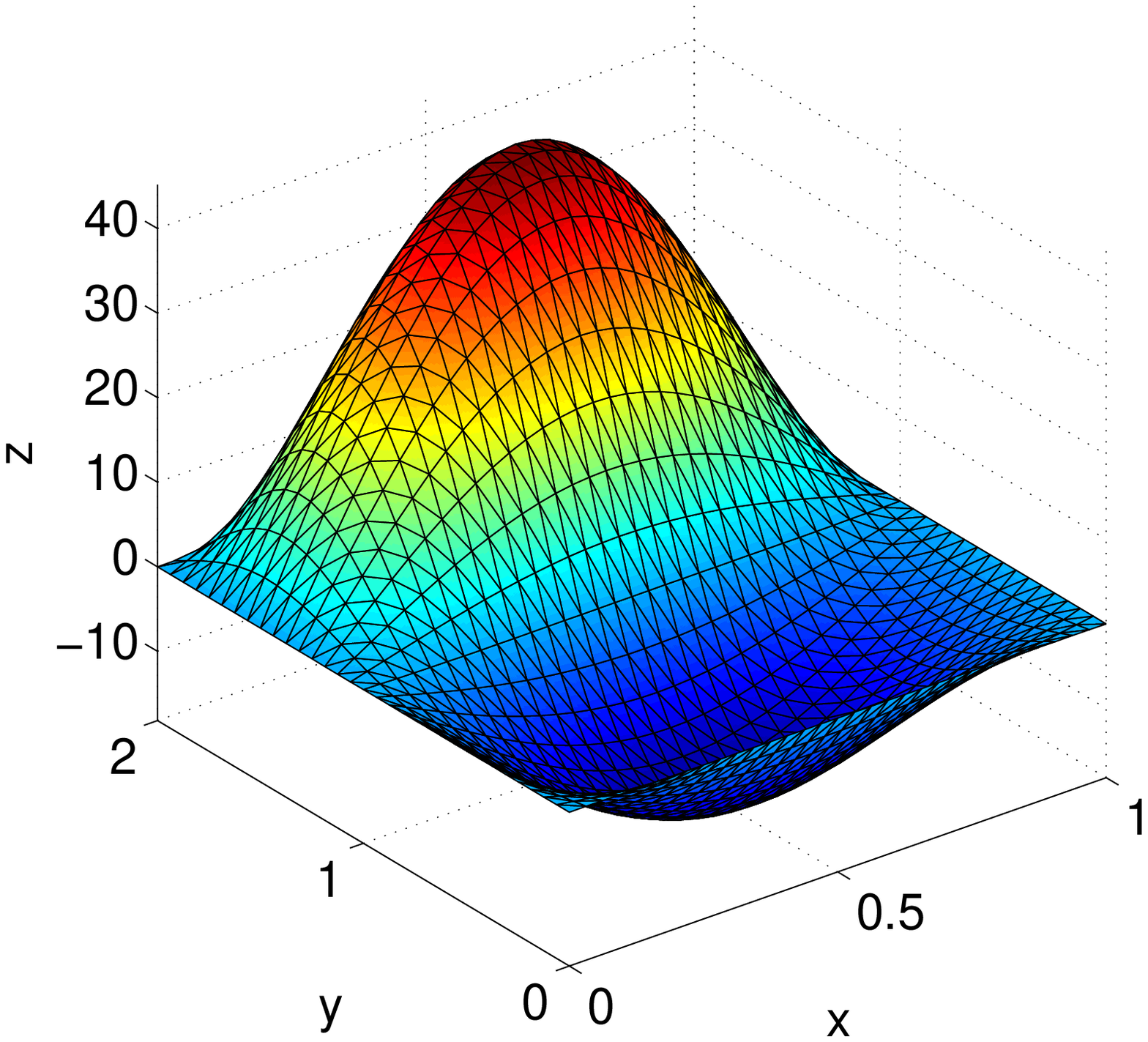}}
\subfigure[$L$ and $R$]{\includegraphics[width=.4\linewidth]{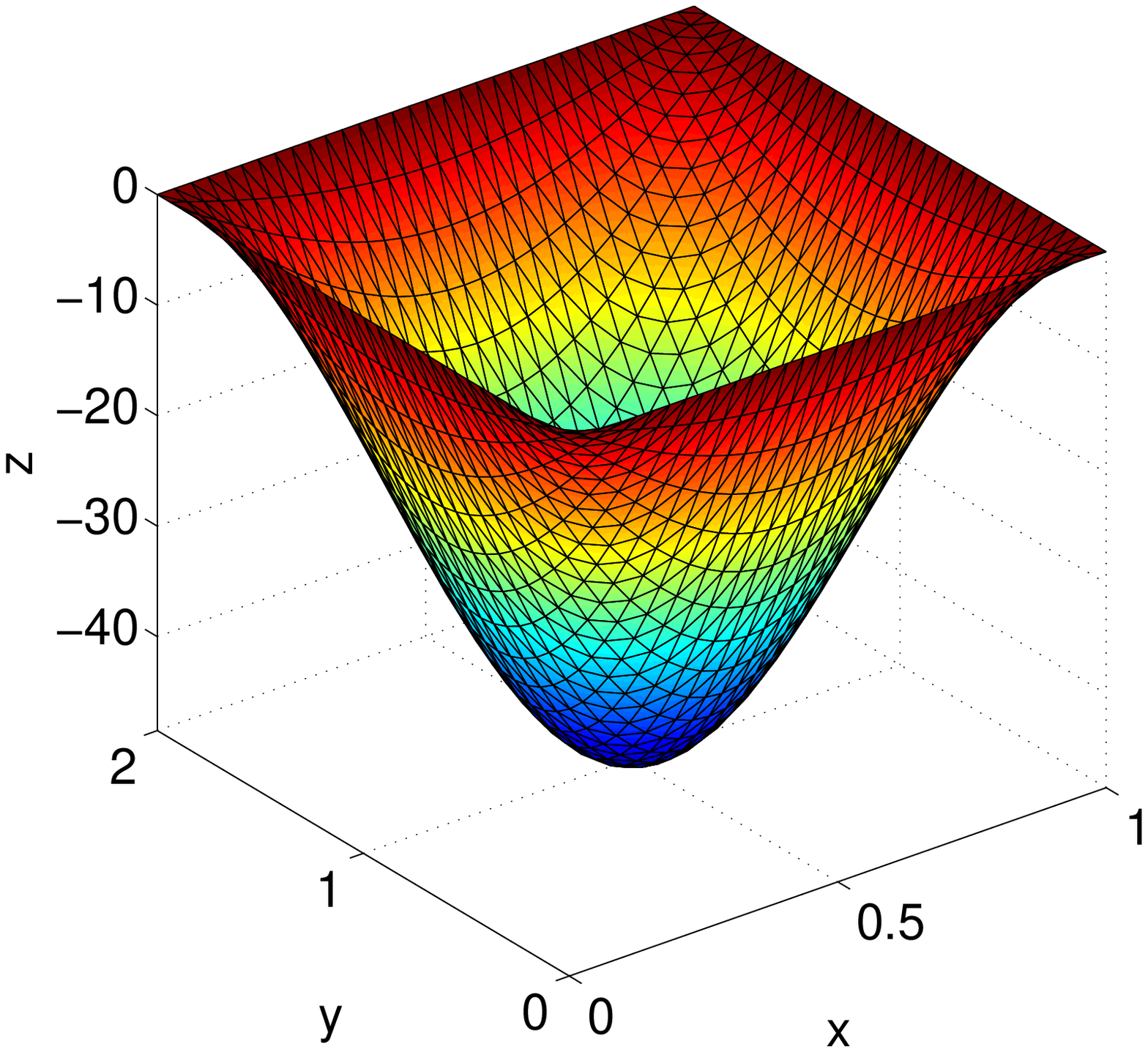}
\includegraphics[width=.4\linewidth]{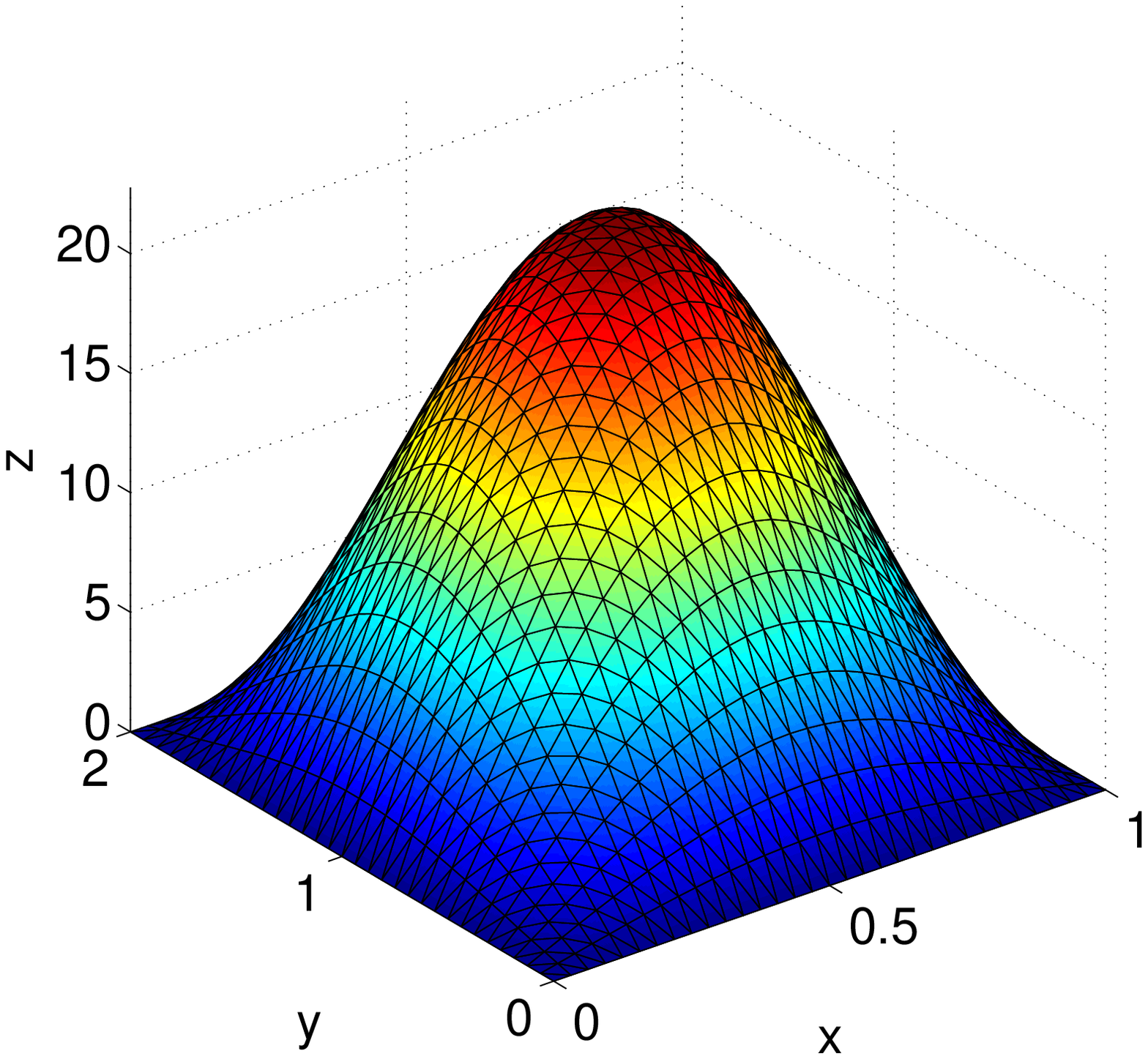}}
\end{center}
\caption{Computed Solutions, 2-D case}
\label{fig:UDLR}
\end{figure}
\bibliography{thesis}
\bibliographystyle{siam}

\end{document}